\documentclass[a4paper]{article}				
\usepackage[latin1]{inputenc}					
\usepackage[T1]{fontenc}						
\usepackage{lmodern}						
\usepackage[english]{babel}					
\usepackage{amsmath,amsfonts}				
\usepackage{amssymb}						
\usepackage{stmaryrd}                      				 
\usepackage{multirow}						
\usepackage{graphicx}
\usepackage{bm}							
\usepackage{amsthm}
\usepackage{mathdots}
\usepackage{algorithm2e}
\usepackage{mathtools}
\usepackage{hyperref}
\usepackage{verbatim}
\usepackage{tikz}
\usetikzlibrary{arrows}
\usepackage[metapost]{mfpic}
\usepackage{url}
\usepackage{bigstrut}
\usepackage{moreverb}
\usepackage[metapost]{mfpic}
\usepackage{tikz}

\usepackage{a4wide} 						

\newtheorem{thm}{Theorem}
\numberwithin{thm}{section}
\newtheorem{lem}[thm]{Lemma}
\newtheorem{lemma}[thm]{Lemma}
\newtheorem{prop}[thm]{Proposition}
\newtheorem{conj}[thm]{Conjecture}
\newtheorem{coro}[thm]{Corollary}
\newtheorem{defi}[thm]{Definition}
\newtheorem{defn}[thm]{Definition}
\newtheorem{qst}[thm]{Question}

\theoremstyle{definition}						
\newtheorem{rem}[thm]{Remark}

\newtheorem{example}[thm]{Example}

\numberwithin{equation}{section}

\newcommand{\R}{\mathbb{R}}
\newcommand{\bbr}{\mathbb{R}}

\newcommand{\N}{\mathbb{N}}
\newcommand{\bbn}{\mathbb{N}}
\newcommand{\Z}{\mathbb{Z}}
\newcommand{\bbz}{\mathbb{Z}}

\newcommand{\F}{\mathbb{F}}
\newcommand{\bbf}{\mathbb{F}}
\newcommand{\K}{\mathbb{K}}
\newcommand{\bbk}{\mathbb{K}}

\newcommand{\bone}{\mathbf{1}}

\newcommand{\bn}{\mathbf{n}}
\newcommand{\br}{\mathbf{r}}
\newcommand{\bk}{\mathbf{k}}
\newcommand{\bmf}{\mathbf{m}}
\newcommand{\btwo}{\mathbf{2}}
\newcommand{\bMf}{\mathbf{M}}
\newcommand{\bs}{\mathbf{s}}
\newcommand{\bj}{\mathbf{j}}
\newcommand{\bl}{\mathbf{l}}
\newcommand{\be}{\mathbf{e}}
\newcommand{\bo}{\mathbf{o}}

\newcommand{\sep}{\;:\;}

\newcommand{\RR}{\K\left(\left(t^{-1}\right)\right)}

\DeclarePairedDelimiter\ceil{\lceil}{\rceil}

\makeatletter

\renewcommand*{\eqref}[1]{%
  \hyperref[{#1}]{\textup{\tagform@{\ref*{#1}}}}%
}

\def\imod#1{\allowbreak\mkern10mu({\operator@font mod}\,\,#1)}
\makeatother

\title{On the $t$--adic Littlewood Conjecture}
\author{Faustin Adiceam, Erez Nesharim \& Fred Lunnon} 
\date{}

\begin{document}
\maketitle

\begin{abstract}
The $p$--adic Littlewood Conjecture due to De Mathan and Teuli\'e asserts that for any prime number $p$ and any real number $\alpha$, the equation $$\inf_{|m|\ge 1} |m|\cdot |m|_p\cdot |\langle m\alpha \rangle|\, =\, 0 $$ holds. Here, $|m|$ is the usual absolute value of the integer $m$, $|m|_p$ its $p$--adic absolute value and $ |\langle x\rangle|$ denotes the distance from a real number $x$ to the set of integers. This still open conjecture stands as a variant of the well--known Littlewood Conjecture. In the same way as the latter, it admits a natural counterpart over the field of formal Laurent series $\mathbb{K}\left(\left(t^{-1}\right)\right)$ of a ground field  $\mathbb{K}$. This is the so--called \emph{$t$--adic Littlewood Conjecture} ($t$--LC).

It is known that $t$--LC fails when the ground field $\mathbb{K}$ is infinite. The present article is concerned with the much more difficult case when this field  is finite. More precisely, a \emph{fully explicit} counterexample is provided to show that $t$--LC does not hold in the case that $\mathbb{K}$ is a finite field with characteristic 3. Generalizations to fields with characteristics different from 3 are also discussed.

The proof is computer assisted. It reduces to showing that an infinite matrix encoding Hankel determinants of the Paper--Folding sequence over $\F_3$, the so--called Number Wall of this sequence, can be obtained as a two--dimensional
automatic tiling satisfying a finite number of suitable local constraints.
\end{abstract}

\paragraph{}
\begin{center}
\emph{In honorem Christiani Figuli \footnote{Christian Potier.}.}
\end{center}

\setcounter{tocdepth}{2}
\tableofcontents

\section{Introduction}

Let $x$ be a real number. Denote by $|x|$ its usual absolute value and by $|\langle x \rangle|$ its distance to the set of integers.
The famous \emph{Littlewood Conjecture} from the 1930's states that for any two real numbers $\alpha$ and $\beta$, the following equation holds:
\begin{equation}\label{LC}
\inf_{m\neq 0} |m|\cdot |\langle m\alpha \rangle| \cdot | \langle m\beta \rangle| \, = \, 0,
\end{equation}
where the infimum is taken over all non--zero integers. The best--known result towards this conjecture is due to Einsiedler, Katok and Lindenstrauss~\cite{EKL} who established that the set of possible counterexamples has Hausdorff dimension zero. It is, however, not even known whether the pair of quadratic irrationalities $(\alpha, \beta) = (\sqrt{2}, \sqrt{3})$ satisfies~\eqref{LC}.

\paragraph{}
De Mathan and Teuli\'e~\cite{MT04} 
suggested a variant of the Littlewood Conjecture which has since then been known as the \emph{$p$--adic Littlewood Conjecture}. According to the latter, given a prime number $p$ and a real number $\alpha$,
\begin{equation}\label{eq:p-LC}
\inf_{m\neq 0} |m|\cdot | m|_p \cdot | \langle m\alpha \rangle| \, = \, 0.
\end{equation}
Here, $| m|_p$ stands for the $p$-adic absolute value of the integer $m$. Upon writing $m = p^{\nu_p(m)}n$, where $\nu_p(m)$ denotes the $p$--adic valuation of $m$ and where $n$ is an integer, \eqref{eq:p-LC} amounts to the following relation~:
\begin{equation}\label{p-LC}
\inf_{n\neq 0, k\ge 0} |n|\cdot | \langle np^k\alpha \rangle| \, = \, 0.
\end{equation}
(Note that it is not required in~\eqref{p-LC} that the integer $n$ should not be divisible by the prime $p$. It is easy to see that this does not affect the claimed equivalence between the two formulations of the problem.) Akin to the Littlewood Conjecture, it is known thanks to the work of Einsiedler and Kleinbock~\cite{EK07} that the set of possible exceptions to the $p$--adic Littlewood Conjecture has Hausdorff dimension zero.

\paragraph{} For detailed accounts on  the Littlewood and the $p$--adic Littlewood Conjectures, see~\cite{BBEK15, B14, Q09} and the references therein.

\paragraph{}
Both of these conjectures admit natural counterparts over function fields, which have attracted much attention. In order to state them, some terminology and notation are first introduced.

\paragraph{}
Let $\mathbb{K}$ be a field. Denote by $\mathbb{K}\left[t\right]$ the ring of polynomials with coefficients in $\mathbb{K}$, and by $\mathbb{K}\left(t\right)$ the field of rational functions over $\bbk$. The valuation on $\mathbb{K}\left[t\right]$ given by the degree of a polynomial extends to a valuation on  $\mathbb{K}\left(t\right)$ so as to provide an absolute value given by
\begin{equation}\label{absfctfield}
\left|\Theta\right|\, =\, 2^{\deg \Theta}
\end{equation}
for any  $\Theta\in\mathbb{K}\left(t\right)$.
The completion of the field of rational functions is then the field of formal Laurent series 
denoted by $\mathbb{K}\left(\left(t^{-1}\right)\right)$. Explicitly, an element $\Theta\in \mathbb{K}\left(\left(t^{-1}\right)\right)$ can be uniquely expressed as a power series with at most finitely many non--zero coefficients corresponding to positive powers of $t$; that is, it can be uniquely expressed as 
$$\Theta \, =\, \theta_{-h}t^h+\dots + \theta_{-1}t+\theta_0 + \theta_1 t^{-1}+ \theta_2 t^{-2}+\dots ,$$ 
where $\left( \theta_{-i}\right)_{i\ge -h}$ is a sequence in $\mathbb{K}$ such that $\theta_{-h}\neq 0$. The \emph{degree} of the Laurent series $\Theta$ is then the integer $h$ and its \emph{absolute value} the quantity 
$\left|\Theta\right|\, =\, 2^{h}$.
Furthermore, one defines the \emph{ fractional part} of $\Theta$ as $$\langle \Theta \rangle\, =\,  \theta_1 t^{-1}+ \theta_2 t^{-2}+\dots ;$$ that is, as  $\Theta$ minus its \emph{polynomial part} $\theta_{-h}t^h+\dots + \theta_{-1}t+\theta_0$.

\paragraph{}
With the above notation, the \emph{Littlewood Conjecture over Function Fields} (LCFF), due to Da\-ven\-port and Lewis~\cite{DL63}, can be stated in complete analogy with the real case as follows: for any $\Theta$ and $\Phi$ in $\mathbb{K}\left(\left(t^{-1}\right)\right)$, the equation 
$$\inf_{N\neq 0} \left|N\right|\cdot \left|\langle N \Theta\rangle\right|\cdot \left|\langle N \Phi \rangle \right|\, =\, 0$$ 
holds. Here, the infimum is taken over all non--zero elements in  $\mathbb{K}\left[t\right]$. In the same vein, De Mathan and Teuli\'e~\cite{MT04} enunciated  the \emph{$t$--adic Littlewood Conjecture} ($t$--LC), which is the analogue over function fields of the $p$--adic Littlewood Conjecture: for any  $\Theta$ in $\mathbb{K}\left(\left(t^{-1}\right)\right)$, the equation
\begin{equation}\label{tadiclit}
\inf_{N\neq 0, k\ge 0} \left|N\right|\cdot \left|\langle N t^k \Theta\rangle\right|\, =\, 0
\end{equation}
holds. Note that in this statement the variable $t$ plays the role of the prime number $p$ in the real case, which is justified by the fact that it can be viewed as an irreducible element in the ring $\mathbb{K}\left[t\right]$~(\footnote{In a more general version of the conjecture, one may replace $t$ with any irreducible polynomial in $\mathbb{K}\left[t\right]$.}).

\paragraph{}
In the case of LCFF, Davenport and Lewis~\cite{DL63} established that the set of exceptions is never empty when the ground field $\mathbb{K}$ is infinite. Their work was complemented by that of Baker~\cite{B64} and several other authors~\cite{B67, C67, C71, K91} who provided explicit counterexamples in this case. In the other direction, see~\cite{AB07} for explicit constructions of pairs of power series satisfying LCFF.  Similarly, for $t$--LC, De Mathan and Teuli\'e~\cite{MT04} established that the conjecture fails when 
$\mathbb{K}$ is infinite. Bugeaud and De Mathan~\cite{BM08} later provided explicit counterexamples in this case (they also gave examples of power series satisfying the conjecture in any characteristic).

\paragraph{}
Much less is known when $\mathbb{K}$ is finite. 
Interesting results were proved by Einsiedler, Lindenstrauss and Mohammadi~\cite{ELM17} where the positive characteristic analogue of the measure classification results of Einsiedler, Katok and Lindenstrauss~\cite{EKL} and Lindenstrauss~\cite{L06} is carried out. It is not immediately clear whether this work implies that LCFF and $t$--LC can only fail on a set of Hausdorff dimension zero. If true, this would provide a counterpart to the above--mentioned results by Einsiedler, Katok and Lindenstrauss~\cite{EKL} and by Einsiedler and Kleinbock~\cite{EK07} on the sets of exceptions to the Littlewood Conjecture and the $p$--adic Littlewood Conjecture, respectively.

\paragraph{}
The aim of the present work is to fill up this gap. More precisely, the following theorem shows that $t$--LC fails over the  field with three elements $\F_3$. It provides an \emph{explicit} counterexample defined from the Paper--Folding sequence $\left(f_n\right)_{n\ge 1}$. Among many other ways, this sequence (also known as the \emph{Dragon Curve Sequence}) can be defined by setting
\begin{equation}\label{rpfseq}
f_n \,= \, \left\{
    \begin{array}{ll}
        0 & \mbox{if } k\equiv 1 \pmod 4;\\
        1 & \mbox{if } k\equiv 3 \pmod 4,
    \end{array}
\right.
\end{equation}
where $n=2^{\nu_2(n)}k$ is a non--zero integer\footnote{To be precise, the (regular) Paper--Folding sequence is more commonly defined in the literature as the sequence $\left(1-f_n\right)_{n\ge 1}$. This corresponds to the coding defined by $0\mapsto 1$ and $1\mapsto 0$ applied to $\left(f_n\right)_{n\ge 1}$ above. It will be slightly more convenient for us to work with definition~\eqref{rpfseq}.}. For accounts on some of the properties enjoyed by this sequence see~\cite[\S 6.5]{AS03} and~\cite{M-FvdP81}.

\begin{thm}\label{thmrinci}
The $t$--adic Littlewood Conjecture fails over $\F_3$. Indeed, the Laurent series $$\Phi\, = \, \sum_{n=1}^{+\infty}f_n t^{-n},$$ where $\left(f_n \right)_{n\ge 1}$ is the Paper--Folding sequence seen as a sequence defined over $\F_3$, is such that 
$$\inf_{N\neq 0, k\ge 0} \left|N\right|\cdot \left|\langle N t^k \Phi\rangle\right|\, = \, 2^{-4}. $$
\end{thm}

\paragraph{}
If $\Theta$ is a power series in $\F_3\left(\left(t^{-1}\right)\right)$, it can also be seen as an element in $ \F_{q}\left(\left(t^{-1}\right)\right)$, where $q=3^s$ with $s\ge 1$. Furthermore, it follows from~\eqref{absfctfield} 
that its absolute value over the latter field 
equals 
its absolute value over the former field. Combined with the result of De Mathan and Teuli\'e~\cite{MT04} for the case when the ground field is infinite, these two observations lead one to the following corollary:

\begin{coro}
The $t$--adic Littlewood Conjecture fails over any ground field with characteristic 3.
\end{coro}

\section{Reduction of the Problem to the Vanishing of certain Hankel Determinants}\label{reduction}

The results established in this section 
are valid over any ground field $\K$. Let 
$\Theta\in\K\left(\left(t^{-1}\right)\right)$. Given the formulation of $t$--LC, one may 
restrict oneself without loss of generality to the case when the polynomial part of $\Theta$ vanishes. Write
\begin{equation}\label{deftheta}
\Theta\, =\, \sum_{n=1}^{+\infty}\theta_nt^{-n},
\end{equation}
where $\left(\theta_n\right)_{n\ge 1}$ is a sequence in $\K$ which, in what follows, 
is identified with the power series $\Theta$ itself.

Define the infinite Hankel matrix $H_{\Theta}$ formed from the power series $\Theta$ as the matrix $H_{\Theta}\, =\,\left(\theta_{i+j-1}\right)_{i,j\ge 1}$; that is, as
\begin{equation*}
	H_{\Theta}\, =\,
	\begin{pmatrix}
		\theta_{1} &  \cdots &  \theta_{k} & \cdots  & \theta_n & \cdots\\
		\vdots &  \iddots &  \vdots & \iddots & \vdots & \iddots \\
		  \theta_{k} &  \cdots &   \theta_{n} & \cdots & \theta_{p} & \cdots  \\
		\vdots &  \iddots &  \vdots & \iddots & \vdots& \iddots  \\				
		\theta_{n} & \cdots & \theta_{p}  & \cdots  & \theta_{q} & \cdots  \\
		\vdots & \iddots & \vdots & \iddots & \vdots & \iddots
	\end{pmatrix} .
\end{equation*}	

Given indices $n\ge 1$ and  $l, m\ge 0$, let 
$H_{\Theta}(n;l,m)$ denote the finite rectangular $(l+1)\times(m+1)$ truncation of the previous matrix with top--left entry $\theta_n$; that is,
\begin{equation*}\label{hankelstructure}
	H_{\Theta}(n;l,m)\, =\,
	\begin{pmatrix}
		\theta_{n} &  \cdots &  \theta_{n+m} \\
		\vdots &  \iddots &  \vdots\\
		  \theta_{n+m} &  \cdots &  \theta_{n+2m}\\
		\vdots &  \vdots  &  \vdots \\	
		  \theta_{n+l-m} &  \cdots &  \theta_{n+l}\\
		\vdots &  \iddots &  \vdots \\				
		\theta_{n+l} & \cdots & \theta_{n+l+m}
	\end{pmatrix}
\end{equation*}	
(this representation corresponds to the case that $l\ge m$ and is easily adapted to the case that $m\ge l$). When $l=m$, one will more conveniently set $$H_{\Theta}(n;l,l) = H_{\Theta}(n;l).$$

\begin{defi}\label{defdeficiency}
Let $\delta\ge 2$ be an integer. The sequence $\left(\theta_n\right)_{n\ge 1}$ of elements in $\K$ is said to have \emph{deficiency} $\delta$ if there exists integers $n\ge 1$ and $l\ge 0$ such that the $\delta-1$ matrices
\begin{equation}\label{deficiencymatr}
H_{\Theta}(n;l),\; H_{\Theta}(n;l+1),\;\dots \;, H_{\Theta}(n;l+\delta-3), \; H_{\Theta}(n;l+\delta-2)
\end{equation}
are singular but such that in any sequence of $\delta$ matrices of the form $$H_{\Theta}(n;l),\;  H_{\Theta}(n;l+1),\dots \;,  H_{\Theta}(n;l+\delta-2),\; H_{\Theta}(n;l+\delta-1)$$ (where $n\ge 1$ and $l\ge 0$), at least one of them is non--singular.

If, for any $n\ge 1$ and $l\ge 0$, none of the matrices $H_{\Theta}(n;l)$ is singular, the sequence  $\left(\theta_n\right)_{n\ge 1}$  is said to have deficiency 1. It is said to have unbounded deficiency if for any integer $\delta\ge 2$, there exist indices  $n\ge 1$ and $l\ge 0$ such that all the matrices~\eqref{deficiencymatr} are singular.
\end{defi}

Say that two finite square submatrices of $H_{\Theta}$ (corresponding to consecutive row and column indices) are nested if they share the same top left entry and if one can be obtained from the other by the addition of a row at the bottom and a column to the right. With this terminology, the sequence  $\left(\theta_n\right)_{n\ge 1}$ having deficiency $\delta\ge 2$ means that one can find a sequence of $\delta-1$ singular submatrices in  $H_{\Theta}$ such that any two successive elements in this sequence are nested; furthermore, there is no sequence of $\delta$ singular submatrices $H_{\Theta}$ enjoying this property. Unbounded deficiency means that arbitrarily long sequences of singular nested submatrices  exist in $H_{\Theta}$ and deficiency 1 just means that  $H_{\Theta}$ contains no singular submatrix.

\paragraph{}
The following theorem reduces $t$--LC to considerations of deficiency of sequences in $\K$.

\begin{thm}\label{themdeficiency}
Let $\left(\theta_n\right)_{n\ge 1}$ be a sequence in $\K$ identified with the power series $\Theta\in\K\left(\left(t^{-1}\right)\right)$ as in~\eqref{deftheta}. Then $\Theta$ satisfies equation~\eqref{tadiclit} (that is, the $t$--adic Littlewood Conjecture over $\K$ is true for $\Theta$) if and only if the sequence $\left(\theta_n\right)_{n\ge 1}$ has unbounded deficiency.

Furthermore, $\Theta$ has deficiency $\delta\ge 1$ if and only if
\begin{equation}\label{eq:deficiency}
\inf_{N\neq 0, k\ge 0} \left|N\right|\cdot \left|\langle N t^k \Theta\rangle\right|\, = \, 2^{-\delta}.
\end{equation}
\end{thm}

\begin{proof}
Let $N=a_h t^h +\, \dots \, +a_1 t+ a_0$ be a non--zero polynomial of degree $h\ge 0$ with coefficients in $\K$. Let $k\ge 0$ and $l\ge 1$ be integers. Clearly,
\begin{equation}\label{equivbase}
\left|N\right|\cdot \left|\langle N t^k \Theta\rangle\right|\, <\, 2^{-l} \;\;\; \Longleftrightarrow \;\;\; \left|\langle N t^k \Theta\rangle\right|\, <\,  2^{-(l+h)}.
\end{equation}	
The fractional part on the left--hand side of the latter inequality can be expanded as follows:
\begin{align*}
\langle N t^k \Theta\rangle \,  &=\, \Big\langle \left(a_ht^{k+h}+a_{h-1}t^{k+h-1}+\,\dots\, +a_1 t^{k+1}+a_0t^k\right)\times \\
&\qquad\qquad\qquad \qquad\qquad\qquad\qquad \qquad\qquad\quad\left(\theta_1t^{-1}+\theta_2t^{-2}+\,\dots\, +\theta_mt^{-m}+\,\dots\right) \Big\rangle \\
&=\, t^{-1}\cdot\left(a_0\theta_{k+1}+ a_1\theta_{k+2}+a_2\theta_{k+3}+\,\dots\, +a_h\theta_{k+h+1}\right) +\\
&\qquad t^{-2}\cdot\left(a_0\theta_{k+2}+a_1\theta_{k+3}+\,\dots\, +a_{h-1}\theta_{k+h+1}+a_h\theta_{k+h+2}\right) +\\ &\qquad \dots\, +\\
&\qquad t^{-u}\cdot\left(a_0\theta_{k+u}+\,\dots\,+a_h\theta_{k+h+u}\right)+ \\ &\qquad \dots
\end{align*}
The second inequality in~\eqref{equivbase} means that the coefficients of $t^{-1}, \,\dots\, , t^{-(l+h)}$  in the above expansion all vanish. Defining $\bm{a}$ as the transpose of the row vector $\left(a_0, \,\dots\, a_h\right)$, this can be restated as follows:
\begin{equation}\label{matrixsystemsol}
H_{\Theta}\left(k+1; h+l-1, h\right)\cdot \bm{a}\,=\, 0.
\end{equation}
Note that $H_{\Theta}\left(k+1; h+l-1, h\right)$ is a rectangular matrix with dimensions $(h+l)\times (h+1)$. Therefore, equation~\eqref{matrixsystemsol} holds for some non--zero vector $\bm{a}$ if and only if this matrix does not have maximal rank $h+1$. This shows that
\begin{equation}\label{equivdeduite}
\eqref{equivbase} \; \iff \; \textrm{rank}\left(H_{\Theta}\left(k+1; h+l-1, h\right)\right)\, <\, h+1.
\end{equation}

\paragraph{} The remainder of the proof is split into two parts in order to establish the following claim: the first inequality in~\eqref{equivbase} holds for some integers $l\ge 1$ and $k\ge 0$ and some non--zero polynomial $N$ if and only if the sequence $\left(\theta_n\right)_{n\ge 1}$ has deficiency at least $l+1$. The statements in Theorem~\ref{themdeficiency} then clearly follow from this equivalence and from the definition of the deficiency of a sequence.

\paragraph{}  To begin with, assume that the first inequality in~\eqref{equivbase} holds for some integers $l\ge 1$  and $k\ge 0$ and some non--zero polynomial $N$ of degree $h\ge 0$. The argument to prove that the sequence $\left(\theta_n\right)_{n\ge 1}$ has deficiency at least $l+1$ is straightforward: the rank condition~\eqref{equivdeduite} means the $h+1$ columns of the $(h+l)\times (h+1)$ matrix $H_{\Theta}\left(k+1; h+l-1, h\right)$ are linearly dependent. Extend this matrix to the square matrix $H_{\Theta}\left(k+1; h+l-1\right)$ with dimensions $(h+l)\times (h+l)$. Then, consider the submatrices $H_{\Theta}\left(k+1; m\right)$ with $m=h,\, \dots\, , h+l-1$ (that is, the matrices obtained as the successive upper--left square submatrices of $H_{\Theta}\left(k+1; h+l-1\right)$ starting from the one with dimensions $(h+1)\times (h+1)$). These submatrices are all singular as their first $h+1$ columns satisfy the same relation of linear dependency as the columns of the initial rectangular matrix $H_{\Theta}\left(k+1; h+l-1, h\right)$. As they are nested and as there are $l$ of them, this proves the claim concerning the deficiency of the sequence $\left(\theta_n\right)_{n\ge 1}$.

\paragraph{} Conversely, assume that the sequence $\left(\theta_n\right)_{n\ge 1}$ has deficiency at least $l+1$. The goal is to prove that  the first inequality in~\eqref{equivbase} holds for some integer $k\ge 0$ and some non--zero polynomial $N$. This amounts to proving the existence of integers $k,h\ge 0$ such that the rank condition~\eqref{equivdeduite} holds for the given integer $l\ge 1$.
The argument is more involved and requires the following lemma, which can be found in~\cite[\S 10]{G59}. The proof is reproduced here for the sake of completeness as it is rather short.

\begin{lem}\label{grandmacher}
Let $H=\left(c_{i+j-1}\right)_{1\le i,j\le n}$ be an $n\times n$ Hankel matrix with entries in $\K$. Assume that the first $r$ columns of $H$ are linearly independent but that the first $r+1$ columns are linearly dependent (here, $1\le r \le n-1$). Then the principal minor of order $r$, that is, $\det\left(\left(c_{i+j-1}\right)_{1\le i,j\le r} \right)$, does not vanish.
\end{lem}

\begin{proof}
Denote by $C_1, \,\dots\, , C_n$ the columns of the matrix $H$ under consideration. By assumption, $C_1, \,\dots\, , C_r$ are linearly independent whereas $C_{r+1}$ can be expressed as a linear combination of the latter: $$C_{r+1}\, =\, \sum_{s=1}^r \alpha_{s}C_s,$$ where $\alpha_1, \,\dots\, , \alpha_s$ are coefficients in the field $\K$. From the Hankel structure of the matrix, this implies that the entries of the matrix $H$ satisfy the recurrence relation
\begin{equation}\label{recurrcnerelationcoeff}
c_k\, =\, \sum_{s=1}^{r}\alpha_{s}c_{k-r-1+s}
\end{equation}
valid for all $k= r+1,\,\dots\, , r+n$. Write
\begin{equation*}\label{martixredueced}
(C_1, \,\dots\, ,C_r)\, =\,
	\begin{pmatrix}
		c_{1} &  c_2 &  c_{3} & \cdots & c_r \\
		c_2 &  c_3 &  c_4 & \cdots & c_{r+1} \\
		  c_{3} &  c_4 &   c_{5} & \cdots & c_{r+2} \\
		\vdots &  \vdots &  \vdots  &  \vdots  &  \vdots \\		
		\vdots &  \vdots &  \vdots  &  \vdots  &  \vdots \\				
		  c_{n-1} &  c_{n} &  c_{n+1} & \cdots &   c_{r+n-2}  \\
		c_{n} &  c_{n+1} &  c_{n+2} & \cdots &   c_{r+n-1}
	\end{pmatrix} ,
\end{equation*}
which is a matrix of rank $r$ by assumption. It follows from the recurrence relations~\eqref{recurrcnerelationcoeff} that each of the rows of this matrix depends linearly on the preceding $r$ rows, hence on the first $r$ ones. Since the matrix has rank $r$, this implies that the first $r$ rows must be linearly independent; that is, that $\det\left(\left(c_{i+j-1}\right)_{1\le i,j\le r} \right)$ does not vanish, as was to be proved.
\end{proof}

Since it is assumed that the sequence $\left(\theta_n\right)_{n\ge 1}$ has deficiency at least $l+1$, its infinite Hankel matrix $H_{\Theta}$ contains $l$ nested singular submatrices, say $H_{\Theta}\left(k+1; h+ m\right)$, where $k, h\ge 0$ and where $m=0, \, \dots\, l-1$. Since the largest of these matrices, viz.~$H:=H_{\Theta}\left(k+1; h+ l-1\right)$, is singular, its $h+l$ columns are linearly dependent. Since the second largest matrix $H_{\Theta}\left(k+1; h+ l-2\right)$ is also singular, Lemma~\ref{grandmacher} implies that the first $h+l-1$ columns of the initial matrix $H$ cannot be linearly independent. Furthermore, since the  third largest matrix $H_{\Theta}\left(k+1; h+ l-3\right)$  is also singular, the same lemma implies that the  first $h+l-2$ columns of the initial matrix $H$ cannot be linearly independent. An easy induction then shows that the first $h+1$ columns of the matrix $H$ are linearly dependent; that is, that $\textrm{rank}\left(H_{\Theta}\left(k+1; h+l-1, h\right)\right)\, <\, h+1$. The equivalence stated in~\eqref{equivdeduite} then shows that the first inequality in~\eqref{equivbase} holds, as was to be established.

\paragraph{}
This completes the proof of Theorem~\ref{themdeficiency}.
\end{proof}

\begin{rem}\label{rem:normalIndices}
It is worthwhile to mention another shorter but less self--contained proof of Theorem~\ref{themdeficiency}. A \emph{best approximation degree} of $\Theta\in\RR$ is an integer $h$ such that there exists a polynomial $N$ of degree $h$ with
\[
\left|\langle N \Theta\rangle\right| < \inf_{\deg M<h}\left|\langle M \Theta\rangle\right|.
\]
Let $\left(h_m\right)_{m=0}^\infty$ be the sequence of best approximation degrees of $\Theta$ and let $\left(N_m\right)_{m=0}^\infty$ be the associated best approximation polynomials. It is known \cite[\S 9]{dT04} that
\begin{equation}\label{eq:bestApproximation}
\left|\langle N_m \Theta\rangle\right| = 2^{-h_{m+1}}.
\end{equation}
Therefore,
\[
\left| N_m\right| \left|\langle N_m \Theta\rangle\right| = 2^{-\left( h_{m+1}-h_m\right)}.
\]
Given $\Theta\in\RR$, let $h_{k,m}$ be the $m^{\textrm{th}}$ best approximation degree of $t^k\Theta$. Then
\[
\inf_{N\neq0,\,k\geq0}\left| N\right| \left|\langle Nt^k \Theta\rangle\right| = \inf_{m\geq0,\,k\geq0}\left| N_{k,m}\right| \left|
\langle N_{k,m}t^k \Theta\rangle\right| = 2^{-\alpha \left(\Theta\right) },
\]
where
\[
\alpha \left(\Theta\right) = \sup\limits_{m\geq0,\,k\geq0}\left( h_{k,m+1} - h_{k,m}\right).
\]
On the other hand, a \emph{normal index} for $\Theta$ is an integer $h$ such that the matrix $H_\Theta(1;h)$ is invertible. It is a standard fact in the theory of Pad\'e approximation that easily follows from \eqref{eq:bestApproximation} that $h$ is a normal index if and only if $h$ is a best approximation degree  (see, e.g., \cite[Proposition 2]{KTW99} for details). Therefore, by Definition \ref{defdeficiency}, the deficiency of $\Theta$ is precisely the quantity  $\alpha \left(\Theta\right)$ defined above.
\end{rem}

\begin{rem}\label{rem:continuedFraction}
Continuing on the previous remark, note that $h_{k,m+1}-h_{k,m}$ is precisely the degree of the rational fraction $N_{k,m+1}/N_{k,m}$, which is also the degree of the $m^{\textrm{th}}$ partial quotient in the continued fraction expansion of $t^k\Theta$ (see \cite[\S 9]{dT04} for an account on the theory of continued fractions in $\RR$). Theorem \ref{themdeficiency} may then be rephrased in this language as follows: equation~\eqref{eq:deficiency} ~holds if and only if the maximal degree of the partial quotients in the continued fraction expansions of all the Laurent series $\Theta,t\Theta,t^2\Theta, \ldots$ is $\delta$.
\end{rem}

In view of Theorem~\ref{themdeficiency}, Theorem~\ref{thmrinci} becomes an immediate corollary of the following statement, which will be established in the next sections:

\begin{thm}\label{thmderive}
The Paper--Folding sequence $\left(f_n\right)_{n\ge 1}$ has deficiency 4 over $\F_3$.
\end{thm}

Considerations of deficiency are ubiquitous in the literature due to their connections with linear recurrence sequences, Pad\'e approximations and problems of irrationality. The  results known in this topic are nevertheless rather limited.

When $\K=\R$, it is not hard to construct a sequence with deficiency 1 by requiring that it should increase sufficiently fast. The situation turns out to be much more complicated in the case that one has to determine the deficiency of a \emph{given} sequence. The fundamental work by Allouche, Peyri\`ere, Wen and Wen~\cite{APWW98}
establishes, with the help of sixteen recurrence relations, that the \emph{principal minors} of the infinite Hankel matrix of the Thue--Morse sequence never vanish. Coons~\cite{C13} obtained a result with a similar flavour in the case of sequences defined from the sum of the reciprocals of the Fermat numbers. A combinatorial and simpler proof of both of these results was later provided by Bugeaud and Han~\cite{BH14}. Coons and Vrbik~\cite{CV12} also considered the case of the Paper--Folding sequence over $\R$ and established computationally that a large (finite) number of the principal minors of its Hankel matrix does not vanish. For further recent results on the non--vanishing of principal minors for classes of real sequences, the reader is referred to~\cite{B, BZ14, BHWY}.

The problem is even less understood over finite fields. The only reasonably complete result seems to be~\cite{APWW98}, where it is shown that some sequences defined as Hankel determinants of the Thue--Morse sequence over $\F_2$ are 2--automatic. In the same paper are also considered the determinants of the successive nested submatrices sharing a common top left entry in the Hankel matrix of the Thue--Morse sequence. It is proved that the sequence thus defined is periodic and not constantly equal to zero modulo 2. This, however, does not mean that the Thue--Morse sequence has bounded deficiency over $\F_2$, as the period may vary with the choice of the top--left entry. In fact, as the generating function of the Thue--Morse sequence is quadratic over $\bbf_2\left(\left(t^{-1}\right)\right)$ (see, e.g., \cite[Equation (9)]{BZ14}), the arguments presented in Section \ref{sec:charTwo} below together with Theorem~\ref{themdeficiency} above imply that the Thue--Morse sequence has unbounded deficiency over $\F_2$. For related work, see also~\cite{hu}.

A remarkable exception to the current poor understanding of the structure of determinants formed from square submatrices in the infinite Hankel matrix of given sequences is due to Kamae, Tamura and Wen~\cite{KTW99}. These authors  consider the Fibonacci word over an alphabet $\{a, b\}$. Substituting $(a,b)=(1,0)$ or $(a,b)=(0,1)$ determines two real sequences and the corresponding infinite Hankel matrices. Rather explicit formulae are provided in~\cite{KTW99} for the determinant of any square submatrix sitting in these matrices. These formulae imply, in particular, that the Fibonacci sequences defined this way have unbounded deficiency. To the best of the authors' knowledge, this, together with the results of the present paper, constitutes the only cases when the determinantal structure of infinite Hankel matrices is completely elucidated for a given eventually non--periodic sequence.

\paragraph{} In what follows, Theorem~\ref{thmderive} will be proved by introducing the concept of a Number Wall, which is an array containing information about the Hankel determinants formed from a given sequence.

\section{The Number Wall of a Sequence}\label{sec:numberWall}

\subsection{Definition and Properties}\label{defpropnbwall}

Let $\Theta = \left(\theta_n\right)_{n\in\Z}$ be a \emph{doubly--infinite} sequence defined over a field $\K$. The \emph{Number Wall} of this sequence is a two--dimensional array  $S\left(\Theta\right)= \left(S_{m,n}(\Theta)\right)_{m,n\in\Z}$ defined as follows: for any $m, n\in \Z$, $S_{m,n}(\Theta)$ is the Toeplitz determinant
\begin{equation}  \label{defnumberwall}
S_{m,n}(\Theta)\, =\,
\begin{vmatrix}
	\theta_{n} &  \theta_{n+1} & \cdots &  \theta_{n+m-1} & \theta_{n+m} \\
	\theta_{n-1} &  \theta_{n} &  \cdots & \theta_{n+m-2} & \theta_{n+m-1} \\
	\vdots &  \vdots &  \vdots  &  \vdots  &  \vdots \\					
	\theta_{n-m+1} &  \theta_{n-m+2} & \cdots & \theta_{n}  &   \theta_{n+1}  \\
	\theta_{n-m} &  \theta_{n-m+1} &  \cdots & \theta_{n-1}  &   \theta_{n}
\end{vmatrix}
\end{equation}
when  $m\ge 0$ and $n\in\Z$ (each diagonal contains the same entry); $S_{m,n}(\Theta) =1$ when $m=-1$ and $n\in\Z$, and  $S_{m,n}(\Theta) =0$ when $m<-1$ and $n\in\Z$. In the Number Wall $ S\left(\Theta\right)$,
\begin{center}
\textit{rows and columns are indexed as in standard matrix notation;}
\end{center}
that is, the first index $m$ records the row number and increases towards page bottom and the second index $n$ records the column number and increases towards page right. This convention will be adopted throughout the paper for \emph{any} array of numbers in the plane.

A Number Wall records Toeplitz determinants of a sequence rather than its Hankel determinants. This enables one to express the properties of such a Wall (see below) in a symmetrical way (this insight is owed to John Conway). Since a Toeplitz determinant as above is obtained under reflection and sign change $(-1)^{m(m+1)/2}$ from the corresponding Hankel determinant, one of these two determinants vanishes if and only if so does the other.

For the sake of simplicity of notation, set from now on $$ S= S\left(\Theta\right) \quad \mbox{ and }\quad S_{m,n}\, =\, S_{m,n}(\Theta)$$ for any $m,n\in\Z$. Properties of Number Walls have been extensively studied in~\cite{L01}. The most fundamental of them, which turns out to be a particular case of the Desnanot--Jacobi identity for determinants, can be stated as follows:

\begin{thm}
For any $m,n\in\Z$, $$S_{m,n}^2\, =\, S_{m+1, n}S_{m-1, n} + S_{m, n+1}S_{m, n-1}.$$
\end{thm}

\begin{proof}
See~\cite[p.8]{L01}.
\end{proof}

It readily follows from this theorem that the entry $S_{m,n}$ in row $m$ in  the Number Wall can be computed from entries in rows $m-1$ and $m-2$ \emph{provided that } $S_{m-2, n}$ \emph{does not vanish}. In the case, however, that the this quantity vanishes, the above formula cannot be used anymore. A remarkable feature of Number Walls is that such zero entries can only occur in very specific shapes:

\begin{thm}\label{themzeroentries}
Zero entries in a Number Wall can only occur within \emph{windows}; that is, within square regions with horizontal and vertical edges.
\end{thm}

\begin{proof}
See~\cite[p.9]{L01}.
\end{proof}

For the sake of brevity, a window in a Number Wall containing only zero entries will from now on be referred to as a \emph{window}. In what follows, it will be convenient to define (with a slight abuse of terminology) the deficiency of a window in a Number Wall as being equal to $\delta\ge 1$ if the window has side length $\delta-1$.

Figure~\ref{[origwind]} below depicts such window. The entries surrounding such a region (here cor\-res\-pon\-ding to the sequences $A_k, B_k, C_k$ and $D_k$) will be referred to as the \emph{inner frame} of the window. The entries surrounding the inner frame (here corresponding to the sequences $E_k, F_k, G_k$ and $H_k$) will be referred to as the \emph{outer frame} of the window.

\par

\begin{figure}[htb]
\centering
{\def\horiz{\leaders\hrule\hfil} 
\def\horil{\tab\leaders\hrule\hfil}
\def\horir{\leaders\hrule\hfil\hbox{}\tab}
\def\vertl{\tab\vrule\hfil}
\def\vertr{\hfil\vrule\tab}
\def \tab {\hskip 12pt}
\def \box {\hskip 24pt}
\vbox{$
\hbox{\vbox{\offinterlineskip 
\halign{&#&#&#&$#$&$#$&#&#&$#$&$#$&$#$&$#$&$#$&$#$&#&$#$&$#$&#&#&#\cr 
&\horil &\horiz &\horiz &\horiz &\horiz &\horiz &\horiz &\horiz &\horiz &\horiz &\horiz &\horiz &\horiz &\horiz &\horiz &\horiz &\horir &\cr
&\vertl &\tab   &\box   &\box   &\hfil  &\hfil  &\box   &\box   &\box   &\box   &\box   &\box   &\hfil  &       &       &\tab   &\vertr &\cr
&\vertl &\tab   &\strut &       &\hfil  &\hfil  &       &       &       &       &       &       &\hfil  &       &       &\tab   &\vertr &\cr
&\vertl &\tab   &       &E_0    &\hfil  &\hfil  &E_1    &E_2    &\ldots &E_k    &\ldots &E_{\delta-1}&\hfil  &E_{\delta}    &       &\tab   &\vertr &\cr
&\vertl &\tab   &\strut &       &\hfil  &\hfil  &       &       &       &       &       &       &\hfil  &       &       &\tab   &\vertr &\cr
&\vertl &\tab   &F_0    &B_0,A_0  &\hfil  &\hfil  &A_1    &A_2    &\ldots &A_k    &\ldots &A_{\delta -1}&\hfil  &A_{\delta},C_{\delta}\ &G_{\delta}    &\tab   &\vertr &\cr
&\vertl &\tab   &\strut &       &\hfil  &\hfil  &       &       &       &       &       &       &\hfil  &       &       &\tab   &\vertr &\cr
&\vertl &\tab   &       &\hfil  &\horil &\horiz &\horiz &\horiz &\horiz &\horiz &\horiz &\horiz &\horir &\hfil  &       &\tab   &\vertr &\cr
&\vertl &\tab   &\strut &       &\vertl &\tab   &       &       &       &       &       &       &\vertr &       &       &\tab   &\vertr &\cr
&\vertl &\tab   &F_1    &B_1    &\vertl &\tab   &{\bf 0}&{\bf 0}&\ldots &{\bf 0}&\ldots &{\bf 0}&\vertr &C_{\delta-1}&G_{\delta-1}&\tab   &\vertr &\cr
&\vertl &\tab   &\strut &       &\vertl &\tab   &       &       &       &       &       &       &\vertr &       &       &\tab   &\vertr &\cr
&\vertl &\tab   &F_2    &B_2    &\vertl &\tab   &{\bf 0}&\ddots &(P)    &\rightarrow&   &\vdots &\vertr &\vdots &\vdots &\tab   &\vertr &\cr
&\vertl &\tab   &\strut &       &\vertl &\tab   &       &       &       &       &       &       &\vertr &       &       &\tab   &\vertr &\cr
&\vertl &\tab   &\vdots &\vdots &\vertl &\tab   &\vdots &(Q)    &\ddots &   &\,\uparrow &{\bf 0}&\vertr &C_k    &G_k    &\tab   &\vertr &\cr
&\vertl &\tab   &\strut &       &\vertl &\tab   &       &       &       &       &       &       &\vertr &       &       &\tab   &\vertr &\cr
&\vertl &\tab   &F_k    &B_k    &\vertl &\tab   &{\bf 0}&\,\downarrow & &\ddots &(R)    &\vdots &\vertr &\vdots &\vdots &\tab   &\vertr &\cr
&\vertl &\tab   &\strut &       &\vertl &\tab   &       &       &       &       &       &       &\vertr &       &       &\tab   &\vertr &\cr
&\vertl &\tab   &\vdots &\vdots &\vertl &\tab   &\vdots &   &\leftarrow &(S)    &\ddots &{\bf 0}&\vertr &C_2    &G_2    &\tab   &\vertr &\cr
&\vertl &\tab   &\strut &       &\vertl &\tab   &       &       &       &       &       &       &\vertr &       &       &\tab   &\vertr &\cr
&\vertl &\tab   &F_{\delta-1}&B_{\delta-1}&\vertl &\tab   &{\bf 0}&\ldots &{\bf 0}&\ldots&{\bf 0} &{\bf 0}&\vertr &C_1    &G_1    &\tab   &\vertr &\cr
&\vertl &\tab   &\strut &       &\vertl &\tab   &       &       &       &       &       &       &\vertr &       &       &\tab   &\vertr &\cr
&\vertl &\tab   &       &\hfil  &\horil &\horiz &\horiz &\horiz &\horiz &\horiz &\horiz &\horiz &\horir &\hfil  &       &\tab   &\vertr &\cr
&\vertl &\tab   &\strut &       &\hfil  &\hfil  &       &       &       &       &       &       &\hfil  &       &       &\tab   &\vertr &\cr
&\vertl &\tab   &F_{\delta}    &B_{\delta},D_{\delta}  &\hfil  &\hfil  &D_{\delta-1}&\ldots &D_k    &\ldots &D_2    &D_1    &\hfil  &D_0,C_0\ &G_0    &\tab   &\vertr &\cr
&\vertl &\tab   &\strut &       &\hfil  &\hfil  &       &       &       &       &       &       &\hfil  &       &       &\tab   &\vertr &\cr
&\vertl &\tab   &       &H_{\delta}   &\hfil  &\hfil  &H_{\delta-1}&\ldots &H_k    &\ldots &H_2    &H_1    &\hfil  &H_0    &       &\tab   &\vertr &\cr
&\vertl &\tab   &\strut &       &\hfil  &\hfil  &       &       &       &       &       &       &\hfil  &       &       &\tab   &\vertr &\cr
&\horil &\horiz &\horiz &\horiz &\horiz &\horiz &\horiz &\horiz &\horiz &\horiz &\horiz &\horiz &\horiz &\horiz &\horiz &\horiz &\horir &\cr
}}}$}}
\caption{A Number Wall Window.}
\label{[origwind]}
\end{figure}
\par

Extending in the natural way the definition of the deficiency of a one--sided sequence (see Definition~\ref{defdeficiency}) to a doubly infinite sequence $\Theta$, this concept can be reinterpreted in terms of some properties satisfied by the Number Wall $S$ of $\Theta$.

To this end, note first that, with the notation of Section~\ref{reduction}, the determinant of the Hankel matrix $H_{\Theta}(n; l)$ (where $n\in\Z$ and $l\ge 0$) is, up to a possible change of sign, the entry $S_{n+l, l}$ of the Number Wall $S$. It thus follows from Definition~\ref{defdeficiency} that $\Theta$ has deficiency $\delta\ge 2$ if and only if its Number Wall admits a diagonal with $\delta-1$ zero entries (that is, if and only if there exist integers $n\in\Z$ and $l\ge 0$ such that $S_{n+l+k, l+k} =0$ for all $0\le k\le \delta-2$) but no diagonal with zero entries of any longer size (the deficiency is equal to 1 if there is no zero entry in the Number Wall). In view of Theorem~\ref{themzeroentries}, this amounts to claiming that the Number Wall of the sequence $\Theta$ admits a window with side length $\delta-1$ but no zero window of any larger size (in the case that the deficiency is 1, this amounts to claiming that there are no zeros at all in the Number Wall).


\begin{thm}\label{innerframe}
The inner frame of a window with finite deficiency $\delta\ge 2$ comprises four geometric sequences, along top, left, right and bottom edge with ratios $P, Q, R$ and  $S$ respectively, from origins at top--left and bottom--right corner (see Figure~\ref{[origwind]}). Furthermore, these ratios satisfy the relation $$\frac{PS}{QR}\,=\, \left(-1\right)^{\delta -1}.$$
\end{thm}

\begin{proof}
See~\cite[p.11]{L01}.
\end{proof}

\begin{coro}\label{coroinnerframe}
With the notation of Figure~\ref{[origwind]} and Theorem~\ref{innerframe}, the inner frame sequences (denoted by  $A_k$, $B_k$, $C_k$ and $D_k$) satisfy the relation $$\frac{A_kD_k}{B_kC_k}\,=\, \left(-1\right)^{(\delta-1)k}$$ for any $0\le k\le \delta$.
\end{coro}

\begin{proof}
See~\cite[p.11]{L01}.
\end{proof}

\begin{thm}\label{outerframe}
With the notation of Figure~\ref{[origwind]} and Theorem~\ref{innerframe}, the outer frame sequences (denoted by $E_k$, $F_k$, $G_k$ and $H_k$) lying immediately outside the inner frame sequences (denoted by  $A_k$, $B_k$, $C_k$ and $D_k$ respectively)  and aligned with them satisfy the relation $$\frac{QE_k}{A_k}+\left(-1\right)^k\frac{PF_k}{B_k}\, =\, \frac{RH_k}{D_k} + \left(-1\right)^k\frac{SG_k}{C_k}$$ for $0\le k \le \delta$.
\end{thm}

\begin{proof}
See~\cite[p.11]{L01}.
\end{proof}

The relations above constitute the set of \emph{frame constraints} of a Number Wall. They can be rephrased all together as follows:

\begin{coro}[Frame Constraints]\label{frameconstraints}
Given a doubly infinite sequence $\left(\theta_n\right)_{n\in\Z}$ over a ground field $\K$, its Number Wall $S=\left(S_{m,n}\right)_{m,n\in\Z}$ is generated by a recurrence 
in row $m\in\Z$ in terms of the previous rows. More precisely, with the notation of Figure~\ref{[origwind]} and Theorem~\ref{innerframe}, given $m,n\in\Z$,
\begin{equation*}
S_{m,n} =
  \begin{cases}
      0 & {\rm if\ } m < -1 ; \\
      1 & {\rm if\ } m = -1 ; \\
      \theta_n & {\rm if\ } m = 0 ; \\
	0 & {\rm if\ } (m,n) {\rm \ is \ within \ a \ window}; \\
      \cfrac{\left(S_{m-1,n}^2 - S_{m-1,n+1} S_{m-1,n-1}\right)} {S_{m-2,n}}
          & {\rm if\ } m > 0 \ \&\ S_{m-2,n} \ne 0 ; \\
      D_k = \cfrac{(-1)^{(\delta-1)k} B_k C_k}{A_k}
          & {\rm if\ } m > 0 \ \&\ S_{m-2,n} = 0 \ \&\ S_{m-1,n} = 0 ; \\
      H_k = \cfrac{Q E_k/A_k + (-1)^k P F_k/B_k - (-1)^k S G_k/C_k}{R/D_k}  & {\rm if\ } m > 0 \ \&\ S_{m-2,n} = 0 \ \&\ S_{m-1,n} \ne 0.
    \end{cases}
\end{equation*}
(In the last two equations above, the index $k$ is determined in the natural way from $m,n$ and $\delta$.)

Conversely, an array 
satisfying this recurrence is the Number Wall of the sequence determined by $\theta_n=S_{0,n}$ for all $n\in\Z$.
\end{coro}

\begin{proof}
This follows immediately from Theorem~\ref{innerframe}, Corollary~\ref{coroinnerframe} and Theorem~\ref{outerframe}. Note that all denominators are guaranteed to be non--zero. Also note that the parameters $k$, $\delta$, $P, Q, R$ and $S$ required for these equations to be a recurrence relation are well--determined from previous rows when not within a window (cf. Figure \ref{fig:algo1} below for the actual algorithm). 
For example, a string of consecutive zeros appearing in a row by application of the equation in the fourth line provides the value of the deficiency $\delta$ of the window and thus of the parameters $P, Q$ and $R$ (e.g., by application of the same equation to compute the vertical entries surrounding the window). Theorem~\ref{innerframe} then gives the value of the parameter $S$.
\end{proof}

The Frame Constraints expressed in Corollary~\ref{frameconstraints} provide a necessary and sufficient condition for an infinite array to be the Number Wall of a sequence. They also give a \emph{Wall Builder Algorithm} for generating finite segments of a given Number Wall as well as an alternative method for verifying its properties. This algorithm is more efficient than a direct use of the definition of a Number Wall from Toeplitz determinants as in~\eqref{defnumberwall}. 
Indeed, a key feature of the Frame Constraints is that in the case of sequences with bounded deficiency, they are \emph{local}. Explicitly, this means that if the sequence in the zeroth row $\left\{S_{0,n}\right\}_{n\in\bbz}$ has deficiency $\delta\ge 1$, then the entry $S_{m,n}$ on row $m$ and column $n$ can be obtained from at most  $O\left(\delta^2\right)$ other entries sitting in previous rows, which number of entries is independent of $m$ and $n$ (see Section~\ref{buildingfinite} below for further details).
By contrast, Definition~\ref{defnumberwall} is global in the sense that $S_{m,n}$ depends on $2m+1$ entries on row 0;  namely, on $S_{0,n'}$ with $\left|n'-n\right|\leq m$.

\subsection{On the Number Wall of the Paper--Folding Sequence over $\F_3$}\label{strategy}

Extend the Paper--Folding sequence to a doubly infinite sequence $\left(f_n\right)_{n\in\Z}$ by setting $f_0=0$ and by defining $f_n$ for $n<0$
via formula~\eqref{rpfseq}, where the integer $k$ is then negative (in other words, $f_{-n}=1-f_n$ for any $n\neq 0$).
	
Note that $f_{0} = f_{1}  = f_{2}=0$. Consider then  the $3\times 3$ Hankel matrix  formed from these three values together with $f_{3}=1$ and $f_{4}=0$; that is, consider the matrix
\begin{equation*}
	\begin{pmatrix}
		0 &  0 &  0 \\
		0 &  0 &  1\\
		0 &  1 &  0
	\end{pmatrix}.
\end{equation*}	

Clearly, all three principal minors of this matrix vanish, which shows that the Paper--Folding sequence has deficiency at least 4. The remainder of the proof will thus consist in proving that the Number Wall of this sequence contains no $4\times 4$ zero block.

Evidence for this conjecture is provided by Figure~\ref{[wallseg]} below, which represents the portion of the Number Wall $\left(S_{m,n}\right)_{m,n\in\Z}$ under consideration in the ranges of indices $-2\le m \le 39$ and $-41\le n \le 41$ (the squares and circles in the figure will be interpreted later).
\par

\begin{figure}[htb!]
  \centering
	\includegraphics[scale=0.55]{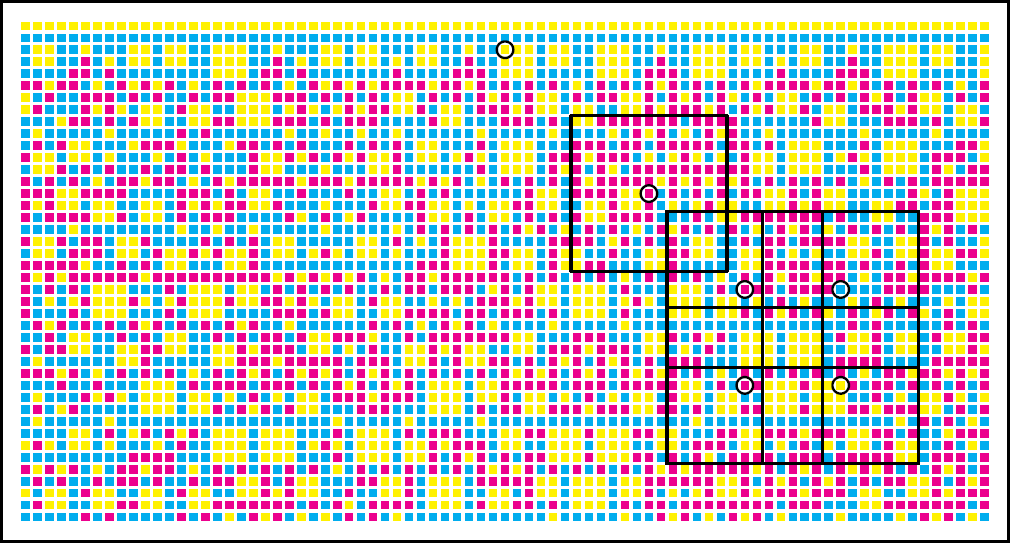}
  \caption{The Paper--Folding Number Wall over $\F_3$ in the ranges $-2\le m\le 39 $ and $-41\le n\le 41$.  Yellow (resp.~blue, pink) entries equal  0 (resp.~1, 2).
}
  \label{[wallseg]}
\end{figure}

The Paper--Folding sequence is well--known to be 2--automatic. From a theorem by Cobham, this amounts to claiming that it is the image, under a coding, of a fixed point of a 2--substitution (see~\cite[\S 6.3]{AS03} for proofs and definitions). The idea of the proof of Theorem~\ref{thmderive} is that such a rich structure in the sequence should be reflected in its Number Wall. With this in mind, the strategy to establish that the Number Wall of the Paper--Folding sequence over $\F_3$ admits no window with deficiency bigger than 4 can be described as follows:

\begin{itemize}
\item Build a suitably large segment of the Paper--Folding sequence and a large portion of its Number Wall;
\item Construct a 2--dimensional substitution and a coding such that the generated tiles cover the portion of the Number Wall under consideration;
\item Consider the infinite tiling obtained by these substitution and coding, and show that it is a valid Number Wall (in other words, that it satisfies the Frame Constraints);
\item Check that the sequence generating the Number Wall thus obtained is the Paper--Folding sequence by showing that this sequence sits in row $m=0$.
\end{itemize}

Each of these steps will be detailed in the next sections. Beforehand, some lemmata related to the theory of tilings are proved. They will be needed to implement the above--described strategy.

\section{Tilings of $\mathbb{Z}^d$}\label{sec:tilings}

Let $\Sigma$ be a set. Its elements will be referred to as \emph{tiles}. Fix once and for all an integer $d\geq1$. A \emph{tiling} of $\mathbb{Z}^d$ (resp.~of $\mathbb{N}^d$) over $\Sigma$ is a function $T:\bbz^d\to\Sigma$ (resp.~a function $T:\mathbb{N}^d\to \Sigma$). Referring to a tiling without referencing $\bbz^d$ or $\bbn^d$ will mean either.

This section partially follows~\cite{AS03} in order to recall a special type of tilings that are generated by a finite set of tiles, a substitution rule that allows one to replace tiles with squares of tiles, and a coding function from the set of tiles to a possibly different set of tiles.
Several equivalent characterisations of such tilings are proved in~\cite{oS87}. In particular, it is shown therein that this may be taken as a definition for an \emph{automatic tiling}, some properties of which are established in this section. These properties will enable one to show in Section \ref{sec:putative} that the Number Wall of the Paper--Folding sequence is an automatic $\bbz^2$--tiling, which will be used to establish Theorem \ref{thmderive}.

Some notation is first introduced.  Boldface letters such as $\bn$ will denote vectors whose coordinates $(n_1,\ldots,n_d)$ are integers or positive integers, as should be clear from the context.  
Let 
$\lceil x\rceil$ denote the ceiling of a real number $x$ and set $\lceil x\rceil_+=\max\left\{0, \lceil x\rceil\right\}$. Given integers $n\in\Z$ and $l\ge 1$, let $[n]_l$ be the representative in $\{1, \dots, l\}$ of $n$ modulo $l$. These pieces of notation are extended  in the natural way to any integer vector $\bm{n}= \left(n_1,\ldots, n_d\right)$, viz.
\begin{align*}
\left\lceil \bm{n}\right\rceil=\left(\lceil n_1\rceil,\ldots, \lceil n_d\rceil\right), \quad \left\lceil \bm{n}\right\rceil_+=\left(\lceil n_1\rceil_+,\ldots, \lceil n_d\rceil_+\right)  \qquad \textrm{ and }  \quad \left[\bm{n}\right]_l = \left(\left[ n_1\right]_l,\ldots, \left[ n_d\right]_l\right).
\end{align*}
Also, the integer vector $\bm{0}$ (resp.~$\bone$ , $\btwo$, $\bm{3}$) will stand for the vector all of whose components are equal to 0 (resp. equal to 1, to 2, to 3). Lastly, the notation $\bl$ (resp.~$\br$, $\br'$) will be reserved to denote the vector all of whose components are equal to a given integer $l$ (resp.~to a given integer $r$, $r'$). In each of these cases, the dimension of the vectors under consideration will be clear from the context.

\subsection{Substitution Tilings}\label{sub:subs}

Only tilings arising from a special type of substitutions will be required:

\begin{defn}\label{def:subs}
Let $\Sigma$ be a set of tiles and let $k\geq2$ be an integer. A \emph{$k$--substitution} is a map $\varphi: \Sigma\to\Sigma^{\left\{1,\ldots ,k\right\}^d}$. A \emph{uniform substitution} is a $k$-substitution for some $k\geq2$.
\end{defn}

In the above definition, the set $\Sigma^{\left\{1,\ldots ,k\right\}^d}$ is the set of mappings from the set of $d$--tuples with integer entries between $1$ to $k$ to $\Sigma$. Thus, the $k$--substitution $\varphi$ maps each tile to a collection of $k^d$ tiles which can be seen as being arranged in the shape of a $d$--dimensional hypercube.

Given a substitution, one can construct a tiling by applying it again and again and ``stacking up'' shifts of the outcome. One way of doing so is introduced with the help of the following definition:

\begin{defn}\label{def:prolongable}
Assume that $\varphi$ is a $k$--substitution on $\Sigma$ (where $k\ge 2$). 
A tile $s\in\Sigma$ is said to be \emph{prolongable} if $s=\varphi(s)(\bone)$.
\end{defn}

Assume  that $\varphi$ is a $k$--substitution on $\Sigma$ and that $s\in\Sigma$ is prolongable. Define an $\bbn^d$--tiling by the recursive formula
\begin{equation}\label{eq:subsTiling1}
T(\bone)=s
\end{equation}
and
\begin{equation}\label{eq:subsTiling2}
T(\bn)=\varphi\left(T\left(\ceil*{\frac{\bm{n}}{k}}\right)\right)\left(\left[\bm{n}\right]_k\right)
\end{equation}
for any $\bn=\left(n_1,\ldots,n_d\right)\in\bbn^d$. This tiling will be denoted by $T_{\left(\varphi,s\right)}$.

To describe a similar construction of substitution tilings in $\bbz^d$,  split first $\bbz^d$ into \emph{orthants}: for any vector $\bo\in\left\{0,1\right\}^d$, the $\bo^{\textrm{th}}$ orthant of $\bbz^d$ is defined as
\[
\bbz_\bo=\left\{\bn\in\bbz^d \sep n_i\leq0 \;\iff\; o_i=0\right\}.
\]
The following definition extends the above notion of prolongability within orthants of $\bbz^d$:

\begin{defn}\label{def:oProlongable}
Assume that $\varphi$ is a $k$--substitution on $\Sigma$ (with $k\ge 2$). For any vector $\bo\in\left\{0,1\right\}^d$, a tile $s\in\Sigma$ is said to be \emph{$\bo$--prolongable} if $s=\varphi(s)\left(\left[o_1\right]_k,\ldots ,\left[o_d\right]_k\right)$ (recall that $[0]_k=k$).
\end{defn}

Given a $2^d$--tuple $\bs=\left(s_\bo\in\Sigma \sep \bo\in\left\{0,1\right\}^d\right)$ such that $s_\bo$ is $\bo$--prolongable for $\varphi$ 
for any $\bo\in\left\{0,1\right\}^d$, define recursively a $\bbz^d$--tiling $T$
with the help of~\eqref{eq:subsTiling2} and of the following extension of the initial condition~\eqref{eq:subsTiling1}: for every 
$\bo\in\left\{0,1\right\}^d$,
let

\begin{equation}\label{eq:subsTiling3}
T(\bo) = s_\bo.
\end{equation}
The $\Z^d$--tiling thus obtained will be denoted by $T_{(\varphi,\bs)}$.

\begin{example}[The Thue--Morse $\N^2$--tiling]
To get used to the above definitions and notation, consider the tiling introduced in \cite[p.20]{APWW98}. It is a substitution tiling of $\bbn^2$ over $\Sigma=\left\{0,1\right\}$ given by the $2$--substitution

\begin{equation*}
0 \mapsto
\begin{array}{cc}
0 &1 \\
1 & 0
\end{array} ,
\quad
1 \mapsto
\begin{array}{cc}
1 &0 \\
0 & 1
\end{array}
\end{equation*}
applied to the $(1,1)$--prolongable tile $0$.
A finite portion of the tiling $T_{(\varphi,0)}$ can be 
generated as follows:
\begin{equation*}
0\mapsto
\begin{array}{cc}
0 &1 \\
1 & 0
\end{array}
 \mapsto
\begin{array}{cccc}
0 &1 & 1 & 0 \\
1 & 0 & 0 & 1\\
1 & 0 & 0 & 1\\
0 & 1 & 1 & 0
\end{array}
 \mapsto
\begin{array}{cccccccc}
0 & 1 & 1 & 0 & 1 & 0 & 0 & 1\\
1 & 0 & 0 & 1 & 0 & 1 & 1 & 0\\
1 & 0 & 0 & 1 & 0 & 1 & 1 & 0\\
0 & 1 & 1 & 0 & 1 & 0 & 0 & 1\\
1 & 0 & 0 & 1 & 0 & 1 & 1 & 0\\
0 & 1 & 1 & 0 & 1 & 0 & 0 & 1\\
0 & 1 & 1 & 0 & 1 & 0 & 0 & 1\\
1 & 0 & 0 & 1 & 0 & 1 & 1 & 0
\end{array}
\mapsto \cdots
\end{equation*}
\end{example}
Recall here the convention that the positive directions for the vertical and horizontal axes are as for matrices; that is, downwards and rightwards, respectively.

\subsection{Coding of a Tiling}\label{sub:codings}
Given two sets of tiles $\Sigma$ and $\Delta$, a \emph{coding} from $\Sigma$ to $\Delta$ is a map $\tau: \Sigma\to\Delta$. It will be convenient to reinterpret this definition as a $1$--coding so as to fit a more general concept:
\begin{defn}\label{def:coding}
Let $\Sigma$ and $\Delta$ be sets and let $l\ge 1$ be an integer. An \emph{$l$--coding} from $\Sigma$ to $\Delta$ is a map
$\tau: \Sigma\to\Delta^{\left\{1,\ldots ,l\right\}^d}$.
\end{defn}

Given a tiling and a coding, another tiling can be generated by applying the coding to each of the tiles. Let $\Sigma$ and $\Delta$ be sets. Let $\tau$ be an $l$--coding from $\Sigma$ to $\Delta$, and let $T$ be any tiling over $\Sigma$. The image of $T$ under $\tau$ is the tiling over $\Delta$ denoted by $\tau(T)$ and defined as follows: for any integer vector $\bn$ in the domain of $T$,

\begin{equation}\label{eq:codedTiling}
\tau(T)(\bn)=\tau\left(T\left(\ceil*{\frac{\bn}{l}}\right)\right)\left(\left[\bn\right]_l\right).
\end{equation}

A tiling which is the image of a uniform substitution tiling under a $1$--coding is said to be an \emph{automatic tiling}. It can be shown that if $T$ is a uniform substitution and $\tau$ is any $l$--coding then $\tau(T)$ is automatic, i.e., there are a uniform substitution tiliing $T'$ and a $1$--coding $\tau'$ such that $\tau(T)=\tau'(T')$ (this can be proved, e.g., via \cite[Lemma 6.9.1]{AS03}). The reason for introducing Definition \ref{def:coding} is that it will enable the use of an efficient algorithm that detects automaticity.

\begin{example}\label{es:paperFolding}
The doubly--infinite Paper--Folding sequence introduced in \S3.2 is well--known to be generated by the $2$--substitution $\psi$ and the $1$--coding $\rho$ defined in Figure \ref{[dragmorf]} below,  applied to the $0$--prolongable and $1$--prolongable tiles $2$ and $0$, respectively (see \cite[Example 10.3.3]{AS03} for further details). With the previous notation, the Paper--Folding sequence is then $\rho\left(T_{(\psi,(2,0))}\right)$.

\begin{figure}[htb]
\centering
\begin{tabular*}{0.5\textwidth}{@{\extracolsep{\fill}}|ccccc|}
\hline
  & Tile & Substitution $\psi$ & Coding $\rho$ & \bigstrut \\
\hline
  & 0 & 0\ 2 & 0 &  \bigstrut[t] \\
  & 1 & 0\ 3 & 1 &  \\
  & 2 & 1\ 2 & 0 &  \\
  & 3 & 1\ 3 & 1 &  \bigstrut[t] \\
\hline
\end{tabular*}
\caption{A standard substitution and coding that generate the Paper--Folding sequence.}
\label{[dragmorf]}
\end{figure}
\end{example}

\subsection{Consistent Overlaps}\label{sub:overlaps}

In this subsection are proved two statements standing at the heart of our approach. To this end, some additional notation and definitions are first introduced. Given $\bj, \bmf\in\bbz^d$, denote by
\[
[\bj,\bmf]=\left\{\bn\in\bbz^d\sep j_i\leq n_i\leq m_i \text{ for every } 1\leq i \leq d \right\}
\]
the \emph{rectangular shape} with edges parallel to the coordinate axes and with opposite vertices $\bj$ and $\bmf$. It will also be convenient to introduce the following notation to exclude some border values of such a parallelepiped:
\[
(\bj,\bmf]=\left\{\bn\in\bbz^d\sep j_i< n_i\leq m_i \text{ for every } 1\leq i \leq d \right\}.
\]
Finally, define the multiplication of this set by a positive integer $r$ in the natural way:
\[
r\cdot (\bj,\bmf] = (r\bj,r\bmf].
\]

\begin{defn}\label{def:pattern}
Given a set of tiles $\Sigma$, a \emph{rectangular pattern} over $\Sigma$ is a map $P: \left[\bone,\bmf\right]\to\Sigma$ for some $\bmf\in\bbn^d$. In the case that $\left[\bone,\bmf\right] = \left\{1,\ldots ,r\right\}^d$ for some $r\in\N$, the map $P: {\left\{1,\ldots ,r\right\}^d}\to\Sigma$ is referred to as an \emph{$r$--pattern}.

The rectangular pattern $P: \left[\bone,\bmf\right]\to\Sigma$ is contained in $T$, which stands either for a tiling or for another rectangular pattern, if there exists $\bj$ such that
\[
P\left(\bn\right)=T(\bj+\bn)
\]
for every $\bn\in\left[\bone,\bmf\right]$.
\end{defn}

\begin{defn}\label{def:consistentCoding}
Given a tiling or a rectangular pattern $T$, given an $l$--coding $\tau$ and an integer $0\leq r<l$, the coding $\tau$ is said to be \emph{consistent with respect to an overlap of $r$} (or, for short, \emph{$r$--consistent}) if for every integer vector $\bn$, every index $1\leq i\leq d$ and every integer $l-r < r'\leq l$, one has that
\begin{equation}\label{eq:consistentCoding}
\tau\left(T\left(\bn\right)\right)\left(\cdot,\ldots,\cdot,r',\cdot,\ldots,\cdot\right)=
\tau\left(T\left(\bn+\be_i\right)\right)\left(\cdot,\ldots,\cdot,r'-(l-r),\cdot,\ldots,\cdot\right)
\end{equation}
whenever $\bn$ and $\bn+\be_i$ are in the domain of $T$.  In this equation, $r'$ and $r'-(l-r)$ appear in the $i^{\textrm{th}}$ coordinate, the equality is as functions of the remaining $d-1$ variables, and $\be_i$ stands for the  $i^{\textrm{th}}$ vector of the standard basis of $\bbr^d$.
\end{defn}

Figure~\ref{consistentov} illustrates the concept of a coding with overlap.

\begin{figure}[h!]
\begin{center}
\begin{tikzpicture}[scale=0.75]
\fill[gray!20] (-3,0.5) -- (-0.5,0.5) -- (-0.5,3) -- (0.5,3) -- (0.5, 0.5)--(3, 0.5)--(3, -0.5)--(0.5, -0.5)--(0.5,-3)--(-0.5,-3)--(-0.5,-0.5)--(-3, -0.5)--cycle;

\draw [very thick] (-3, 3) -- ++(6,0) -- ++(0,-6) --++ (-6,0) -- cycle;
\draw [very thick] (-0.5,3) -- (-0.5,-3) ;
\draw [very thick] (0.5,3) -- (0.5,-3) ;
\draw [very thick] (-3,0.5) -- (3,0.5) ;
\draw [very thick] (-3,-0.5) -- (3,-0.5) ;

\draw [very thick]  (1.25,1.25) circle (0.1) ;
\draw [very thick]  (-1.25,1.25) circle (0.1) ;
\draw [very thick]  (1.25,-1.25) circle (0.1) ;
\draw [very thick]  (-1.25,-1.25) circle (0.1) ;

\draw[<->,>=latex] (-3, 3.5) -- (0.5,3.5);
\draw (-1.25,3.7) node[above]{$l$} ;
\draw[<->,>=latex] (0.5, -3.5) -- (3,-3.5);
\draw (1.75,-3.7) node[below]{$l-r$} ;
\draw[<->,>=latex] (-3.5, -0.5) -- (-3.5, 0.5);
\draw (-3.7, 0) node[left]{$r$} ;

\end{tikzpicture}
\caption{Schematic representation in the plane of an $l$--coding $\tau$ which is $r$--consistent with respect to an underlying tiling $T$. The circles represent four adjacent tiles in the tiling $T$ to which the coding $\tau$ is applied. The shaded regions correspond to the areas of overlap between two adjacent $l\times l$ squares representing the images under $\tau$ of these tiles.}
\label{consistentov}
\end{center}
\end{figure}

If $T$ is a tiling over $\Sigma$ and $\tau$ an $l$--coding which is $r$--consistent, define an $(l-r)$--coding $\tau_r$ by setting
\begin{equation}\label{eq:defnTauTag}
\tau_r(s)=\tau(s)|_{\left\{1,\ldots,l-r\right\}^d}
\end{equation}
for every $s\in\Sigma$. The resulting tiling $\tau_r(T)$ will be referred to as a \emph{tiling with $l$--coding and overlap $r$}.

For any $\bj\in\{0,\ldots,r\}^d$, let $\tau_{(r,\bj)}$  be the $(l-r)$--coding defined for a given $s\in\Sigma$ as a map $$\tau_{(r,\bj)}(s):\left\{1,\ldots ,l-r\right\}^d\to\Delta$$ such that
\begin{equation}\label{eq:deftaurj}
\tau_{(r,\bj)}(s)(\bmf)=\tau(s)\left(\bj + \bmf\right)
\end{equation}
for all $\bmf\in \left\{1,\ldots ,l-r\right\}^d.$

\begin{lem}\label{eq:shiftedCoding}
With the above notation and definitions, the tiling $\tau_{(r,\bj)}(T)$ satisfies the relation
\[
\tau_{(r,\bj)}(T)(\bn) = \tau_r(T)(\bj + \bn)
\]
for any integer vector $\bn$ such that both of these quantities are well--defined.
\end{lem}

\begin{proof}
Decompose the vectors $\bn$ and $\bj$ as $$\bn = \bn'(l-r)+\bn'' \quad \textrm{and} \quad \bj=\bj'(l-r)+\bj'',$$ where each of the components of the integer vectors $\bn''$ and $\bj''$ lies between 1 and $l-r$. Decompose also the vector $\bn''+\bj''$ as $$\bn''+\bj'' = \mathbf{u}(l-r)+\mathbf{v},$$ where here also each of the components of the vector $\mathbf{v}$ lies between 1 and $l-r$.

Then,
\begin{align*}
\tau_{(r,\bj)}(T)(\bn) \; &\underset{\eqref{eq:codedTiling}}{=} \; \tau_{(r,\bj)}\left(T\left(\ceil*{\frac{\bn}{l-r}}\right)\right)(\left[\bn\right]_{l-r})\\
&\underset{\eqref{eq:deftaurj}}{=} \; \tau\left(T\left(\ceil*{\frac{\bn}{l-r}}\right)\right)(\left[\bn\right]_{l-r}+\bj)\\
&=\; \tau\left(T\left(\bn'+\bone\right)\right)\left(\mathbf{v}+(l-r)(\bj'+\mathbf{u})\right)\\
&\underset{\eqref{eq:consistentCoding}}{=}\;  \tau\left(T\left(\bn'+\bj'+\mathbf{u}+\bone\right)\right)\left(\mathbf{v}\right)\\
&=\; \tau_{r}\left(T\left(\ceil*{\frac{\bn+\bj}{l-r}}\right)\right)(\left[\bj+\bn\right]_{l-r})\\
&\underset{\eqref{eq:codedTiling}}{=}  \; \tau_r(T)(\bj + \bn).
\end{align*}

\end{proof}
Lastly, it will be useful to extend the definition of a coding with overlap to rectangular patterns:

\begin{defn}\label{def:consistentPattern}
Let $P$  be a rectangular pattern on $[\bone,\bmf]$ and let $\tau$ be an $r$--consistent $l$--coding (where $0\leq r < l$). The \emph{image of $P$ under the coding $\tau_r$} is the rectangular pattern $\tau_r(P)$ defined on 
\begin{equation}\label{eq:domain}
[\bone, (l-r)\bmf+\br]
\end{equation}  
by
\begin{equation}\label{eq:codedPattern}
\tau_r(P)(\bn) = \tau\left(P\left(\min\left\{\ceil*{\frac{\bn}{l-r}},\bmf\right\}\right)\right)\left(\left[\bn\right]_{l-r}+(l-r)\cdot \ceil*{\frac{\bn - (l-r)\bmf}{(l-r)}}_+\right).
\end{equation}
Here, the $\min$ function is applied to vectors coordinate--wise.
\end{defn}

Note that if the integer vector $\bn$ is decomposed uniquely as $$\bn=\bn' (l-r)+\bn''$$ with $1\le n_i''\le l-r$, then formula~\eqref{eq:codedPattern} can be rewritten more compactly as
\begin{equation*}
\tau_r(P)(\bn) = \tau\left(P\left(\min\left\{\bn'+\bone,\bmf\right\}\right)\right)\left(\bn''+(l-r)\cdot \ceil*{\bn'-\bmf+\bone}_+\right).
\end{equation*}

This definition is designed to satisfy the following property:

\begin{lemma}\label{prop:tauOfPattern}
If $T$ is a tiling and $\tau$ is an $l$--coding which is $r$--consistent, then for any rectangular pattern $P$ in $T$, it holds that $\tau_r(P)$ is contained in $\tau_r(T)$. More precisely, if $P$ is a pattern on $[\bone,\bmf]$ and $\bj$ is an integer vector such that $P(\bn)=T(\bj+\bn)$ for all $\bn\in [\bone,\bmf]$, then
\[
\tau_r(P)(\bn)=\tau_r(T)\left((l-r)\bj+\bn\right)
\]
for every $\bn\in\left[\bone,(l-r)\bmf+\br\right]$.
\end{lemma}

\begin{proof}
Decompose any given $\bn\in\left[\bone,(l-r)\bmf+\br\right]$ as
\begin{equation}\label{decompos1}
\bn=\bn'(l-r)+\bk
\end{equation}
with $0 \leq n'_i\leq m_i-1$ and
\[
1\leq k_i\leq
    \begin{cases}
      l-r & {\rm if\ } n_i' < m_i-1, \\
      l  & {\rm if\ } n_i' = m_i-1. \\
      \end{cases}
\]
Furthermore, let
\begin{equation}\label{decompos2}
\bk=\bk'(l-r)+\bk''
\end{equation}
with $k'_i\geq 0$ and $1\leq k''_i\leq l-r$ for any $1\leq i\leq d$. Then

\begin{align*}
\tau_r(T)\left((l-r)\bj+\bn\right) \; &\underset{\eqref{eq:codedTiling}}{=}\; \tau_r\left(T\left(\bj + \ceil*{\frac{\bn}{l-r}}\right)\right)\left(\left[\bn\right]_{l-r}\right)\\
&= \,\tau\left(T\left(\bj + \bn' + \bk' + \bone\right)\right)\left(\bk''\right)
\\
&= \,\tau\left(T\left(\bj + \bn' + \bone\right)\right)\left(\bk\right)
\\
&= \,\tau\left(P\left(\bn' + \bone\right)\right)\left(\bk\right)
\\
&= \,\tau_r\left(P\right)\left(\bn\right).
\end{align*}

In these relations, the third equality follows upon applying identity~\eqref{eq:consistentCoding} $k_i'$ times in the $i^{\textrm{th}}$ coordinate for each $1\leq i\leq d$. The last equality follows from Definition~\ref{def:consistentPattern} and a direct calculation upon using decomposition~\eqref{decompos1} in each coordinate where $n'_i<m_i+1$ and upon using decompositions~\eqref{decompos1} and~\eqref{decompos2} in each coordinate where $n'_i=m_i-1$.
\end{proof}

Lemma~\ref{prop:tauOfPattern} enables one to state the first of the two fundamental results of this subsection:

\begin{lemma}\label{lem:consistentOverlap}
Assume that $T$ is a tiling over $\Sigma$ and that $\tau: \Sigma\to\Delta^{\{1,\ldots l\}^d}$ is an $l$--coding which is $r$--consistent.
Then for every integer $r' \geq 1$, every $r'$--pattern in $\tau_{r}(T)$ is contained in the image under the coding $\tau_{r}$ of some $s(l,r,r')$--pattern in $T$, where 
\begin{equation}\label{eq:slrr}
s(l,r,r') =1+\ceil*{\frac{r'-(r+1)}{l-r}}_+.
\end{equation} 
\end{lemma}

\begin{rem}
One could simplify  Definition \ref{def:consistentPattern} and the proofs of Lemmata \ref{prop:tauOfPattern} and \ref{lem:consistentOverlap} by replacing the domain \eqref{eq:domain} with a smaller one, namely, $[\bone, (l-r)\bmf]$. This would come at the cost of replacing \eqref{eq:slrr} with the possibly bigger number $1+\ceil*{\frac{r'-1}{l-r}}_+$. For the purpose of the current work, this would only have a small effect (see the proofs of Lemma \ref{lempatternzero4neg} and Theorem \ref{thm:frameConstraints} for the parameters that are required for the Paper--Folding Number Wall). 
\end{rem}

\begin{proof}
Assume that $P: \left\{1,\ldots ,r'\right\}^d\to \Delta$ is an $r'$--pattern in $\tau_r(T)$  and let $\mathbf{j}$ be an integer vector such that
\begin{equation}\label{defpntr}
P\left(\bn\right)\;=\;\tau_r(T)(\mathbf{j}+\bn)
\end{equation}
for every $\bn\in\left\{1,\ldots ,r'\right\}^d$. Decompose $\mathbf{j}$ as $$\mathbf{j}=\mathbf{j}'(l-r)+\mathbf{j}''$$ with $0\leq j''_i < l-r $ for every $1\leq i\leq d$. Let $P'$ be the $s(l,r,r')$--pattern in $T$ defined by
\[
P'(\bn) = T\left(\mathbf{j}' + \bn\right)
\]
for all $\bn\in\left\{1, \dots, s(l,r,r')\right\}^d$.

It then follows from Lemma~\ref{prop:tauOfPattern} (applied with the $d$-tuple $\bmf$ all of whose components are equal to $s(l,r,r')$) that
\[
\tau_r\left(P'\right)\left(\mathbf{j}''+\bn\right) \; = \; \tau_r\left(T\right)\left(\mathbf{j}+\bn\right) \; \underset{\eqref{defpntr}}{=}\; P\left(\bn\right)
\]
whenever
\[
\mathbf{j}''+\bn\in\left[\bone,(\bl-\br)\cdot s(l,r,r')+\br\right].
\]
This condition holds in the present case. Indeed, since  $\bn\in\left\{1,\ldots ,r'\right\}^d$ and  $j''_i<l-r$ for every $1\leq i\leq d$,
\[
j''_i + n_i \leq l -r - 1 + r' \leq l + (l-r)\cdot\ceil*{\frac{r'-(r+1)}{l-r}}_+ = (l-r)\left(1+\ceil*{\frac{r'-(r+1)}{l-r}}_+\right) + r,
\]
whence the claim from the definition of the quantity $s(l,r,r')$.

This concludes the proof of the lemma.
\end{proof}

Lemma \ref{lem:consistentOverlap} will be used in the proof of Theorem \ref{thmderive}. 
In order to apply it, one first needs to verify that a given coding is consistent. While this can be a hard task in general, the following statement provides an easy--to--check criterion in the case of substitution tilings:

\begin{lemma}\label{lem:completeTetrad}
Let  $T$ be a $k$--substitution tiling (where $k\ge 2$). Assume that there are integers $m_i\leq0 < M_i$ ($1\leq i \leq d$) such that $m_i=0$ if $T$ is an $\bbn^d$--tiling and such that $m_i<0$ if $T$ is a $\bbz^d$--tiling verifying the following condition:   every $2$--pattern contained in $T|_{k\left(\bmf,\bMf\right]}$ is already contained in $T|_{\left(\bmf,\bMf\right]}$.

Then every $2$--pattern in $T$ is already contained in $T|_{\left(\bmf,\bMf\right]}$.
\end{lemma}

Before proving the lemma, extend first  the definition of a $k$--substitution $\varphi$ over a set of tiles $\Sigma$ to rectangular patterns. This can be done in the natural way by setting
\begin{align}\label{sblurf}
\varphi(P)(\bn) = \varphi\left(P\left(\left\lceil\frac{\bn}{k}\right\rceil\right)\right)\left(\left[\bn\right]_k\right)
\end{align}
for all $\bn\in\N^d$.

\begin{proof}
Let $\Sigma$ be the set of tiles and let $\varphi:\Sigma\to\Sigma^{\{1,\ldots,k\}^d}$ be the substitution defining $T$. Let $P: \{1,2\}^d\to\Sigma$ be the $2$--pattern  in $T$ defined by
\begin{equation}\label{defPTjn}
P(\bn)=T(\bj + \bn)
\end{equation}
for some integer vector $\bj$. Assume that $P$ is contained in $T|_{k^l\left(\bmf,\bMf\right]}$ for some $l\ge 1$. It will be shown by induction on the integer $l$ that $P$ also lies in $T|_{\left(\bmf,\bMf\right]}$. The claim being true by assumption when $l=1$, assume that $l\ge 2$ and consider the rectangular pattern
\[
P': \left[\bone,\bone + \ceil*{\frac{\bj+\btwo}{k}} - \ceil*{\frac{\bj+\bone}{k}}\right]\to\Sigma
\]
defined by

\begin{equation}\label{defP'}
P'(\bn)=T\left(\ceil*{\frac{\bj+\bn}{k}}\right).
\end{equation}
Then, $P'$ is contained in a 2--pattern. Indeed, decompose the $i^{th}$ component $j_i$ of $\bj$ as $$j_i = ku_i+v_i \qquad \mbox{ with }\qquad 0\le v_i\le k-1$$ and note that

\begin{equation}\label{1or2forj_i}
1 + \ceil*{\frac{j_i+2}{k}} - \ceil*{\frac{j_i+1}{k}} =
\begin{cases}
2 & \text{ if }  v_i = k-1;\\
1 & \text{ otherwise. }
\end{cases}
\end{equation}

Furthermore, the pattern $P$ is contained in the image under $\varphi$ of the pattern $P'$. Indeed, on the one hand,
 it easily follows from~\eqref{1or2forj_i} that for any $\bn\in \left[\bone,\btwo\right],$ one has that $$\ceil*{\frac{\left[\bj+\bone\right]_k+\bn-\bone}{k}}\in \left[\bone,\bone + \ceil*{\frac{\bj+\btwo}{k}} - \ceil*{\frac{\bj+\bone}{k}}\right].$$

On the other hand,\begin{align*}
\varphi\left(P'\right)\left(\left[\bj+\bone\right]_k+\bn-\bone\right) \; & \underset{\eqref{sblurf}}{=}\;  \varphi\left(P'\left(\ceil*{\frac{\left[\bj+\bone\right]_k+\bn-\bone}{k}}\right)\right)\left(\left[\left[\bj+\bone\right]_k+\bn-\bone\right]_k\right)\\
& \underset{\eqref{defP'}}{=}\; \varphi\left(T\left(\ceil*{\frac{\ceil*{\bj+\frac{\left[\bj+\bone\right]_k+\bn-\bone}{k}}}{k}}\right)\right)\left(\left[\left[\bj+\bone\right]_k+\bn-\bone \right]_k\right) \\
& \underset{\eqref{1or2forj_i}}{=}\; \varphi\left(T\left(\ceil*{\frac{\bj+\bn}{k}}\right)\right)\left(\left[\bj+\bn \right]_k\right) \\
& \underset{\eqref{eq:subsTiling2}}{=}\; T\left(\bj+\bn\right) \\
& \underset{\eqref{defPTjn}}{=}\; P\left(\bn\right),
\end{align*} whence the claim.

By the induction assumption, the pattern $P'$, which is contained in $T|_{k^{l-1}\left(\bmf,\bMf\right]}$ from its definition, already appears in $T|_{\left(\bmf,\bMf\right]}$. Its image under $\varphi$ is therefore contained in the image of $T|_{\left(\bmf,\bMf\right]}$ under $\varphi$, which is clearly contained in $T|_{k\left(\bmf,\bMf\right]}$. Thus, from the above claim, the pattern $P$ also lies in  $T|_{k\left(\bmf,\bMf\right]}$, and therefore in $T|_{\left(\bmf,\bMf\right]}$ by assumption.

This completes the proof of the lemma.
\end{proof}

\begin{coro}\label{latest}
Let $T$ be a tiling satisfying the assumptions of Lemma~\ref{lem:completeTetrad}. Assume that $\tau$ is an $l$--coding satisfying this property: there exists an integer $r$, where $1\leq r < l$, such that the consistency condition~\eqref{eq:consistentCoding} holds for every $\bn$ with $m_i< n_i\leq M_i$.

Then the coding $\tau$ is $r$--consistent for the tiling $T$.
\end{coro}

\begin{proof}
It immediately follows from Lemma~\ref{lem:completeTetrad} that the consistency condition~\eqref{eq:consistentCoding} only needs to be checked in the region $\left(\bmf,\bMf\right]$ for it to hold everywhere.
\end{proof}

\begin{rem}
Properties of substitution tilings with overlaps are also studied in~\cite{L09}, where they are referred to as \emph{C3DEL} systems --- an acronym representing ``Deterministic Lindenmayer system with constant width, inflation, encoding and context''. This terminology extends the one presented in~\cite[\S 7.12]{AS03}.
\end{rem}

\section{Generating the Putative Number Wall of the Paper--Folding Sequence over $\F_3$}\label{sec:putative}

The goal of this section is to explain how to generate an automatic Number Wall that coincides with a large segment of the Paper--Folding Number Wall. Verifying that this
automatic tiling is in fact equal to the Paper--Folding Number Wall, and that it has deficiency 4, is postponed to Section \ref{sec:verify}.

\subsection{Building a Finite Portion of a Number Wall}\label{buildingfinite}

The recursive formula in Corollary \ref{frameconstraints} gives a formula for $S_{m,n}$ in terms of elements occupying previous rows. This might seem impractical at first glance due to infinitely many computations required to calculate each row. However, looking more carefully reveals that the same formula gives an algorithm for calculating an isosceles triangle contained in a trapezoidal portion of the wall of the form
\[
\left(S_{k,l}\right)_{0\leq k\leq m,\;n-m+k\leq l\leq n+m-k}
\]
given the finite portion $\left(S_{0,l}\right)_{n-m\leq l\leq n+m}$ of the initial sequence $\left(S_{0,l}\right)_{l\in\Z}$. The discovery of this algorithm is discussed in \cite{L01}. It is described in Figure \ref{fig:algo1} below.

\begin{figure}[htb!]
\framebox{
\begin{algorithm}[H]
 \KwData{$S_n = S_{0,n}$ for $a\leq n \leq b$}
 \KwResult{$S_{m,n}$ for $0\leq m\leq \frac{b-a}{2}$ and $a+m\leq n\leq b-m$}
 \While{$a\leq n\leq b$}{
  $S_{-2,n}=0$ and $S_{-1,n}=1$\par
 }
\While{$0\leq m\leq\frac{b-a}{2}$ and $a+m\leq n\leq b-m$}{
  \uIf{$m=0$}{
   $S_{m,n}=S_n$ \par
  }
  \uElseIf{$S_{m-2,n}\neq0$}{
    $S_{m,n}=\cfrac{\left(S_{m-1,n}^2 - S_{m-1,n+1} S_{m-1,n-1}\right)} {S_{m-2,n}}$ \par
  }
  \Else{
    Set $p=q=k=1$ \par
    \While{$S_{m-p-2,n}=0$}{
      $p=p+1$\par
      }
    \While{$S_{m-p-1,n-q}=0$ and $n-q\geq a+m-p-1$}{
      $q=q+1$\par
      }
    \While{$S_{m-p-1,n+k}=0$ and $n+k\leq b-m-p-1$}{
      $k=k+1$\par
      }
    $\delta=k+q$\par
    \uIf{$\delta>p+2$}{
      Inside a zero block let $S_{m,n}=0$\par
      }
    \uElseIf{$\delta=p+2$}{
      On the inner frame let $S_{m,n} = \cfrac{(-1)^{(\delta-1)k}S_{m-q,n-q}S_{m-k,n+k}}{S_{m-\delta,n-q+k}}$\par
      }
    \Else{
      On the outer frame where:\par
	  $P = \cfrac{S_{m-\delta-1,n}}{S_{m-\delta-1,n-1}}$\par
      $Q = \cfrac{S_{m-q-1,n-q}}{S_{m-q-2,n-q}}$\par
	$R = \cfrac{S_{m-k-1,n+k}}{S_{m-k,n+k}}$\par
      and
      $S = \cfrac{S_{m-1,n}}{S_{m-1,n+1}},$\par
      let
      $S_{m,n} =  \cfrac{S_{m-1,n}}{R} \left( \cfrac{Q S_{m-\delta-2,n+k-q}}{S_{m-\delta-1,n+k-q}} + (-1)^k\left( \cfrac{PS_{m-q-1,n-q-1}}{S_{m-q-1,n-q}} - \cfrac{SS_{m-k-1,n+k+1}}{S_{m-k-1,n+k}}\right)\right)$\par
      }
  }
}
\vspace{3mm}

\end{algorithm}
}
\caption{How to calculate a portion of the Number Wall given a portion of a sequence.}
\label{fig:algo1}
\end{figure}

\subsection{Tiling a Finite Table with a Substitution, Coding and Overlap}\label{sec:tilingAlgo}

Assume one is given a rectangular segment of an automatic tiling. The goal is to find a substitution and a coding which generate the entire tiling. If the number of tiles in the infinite tiling is bounded and if the given pattern is large enough, this is then possible. For an $\bbn$--tiling, i.e.~for sequences, this is done by Sutner and Tetruashvili~\cite{ST12} via finite automata. By a well--known process described in \cite[Propositions 10.1.5 \& 10.2.2]{MLothaire}, their algorithm also yields a substitution and a coding which generate the entire sequence.
It is possible to extend this approach to higher dimensions. However, such an algorithm might not be efficient enough for our purposes, as the computational cost may be exponential in the number of tiles (see~\cite[p.16]{jS} for further details). In the case of the Paper--Folding Number Wall, this number is in the order of magnitude of several thousands.

A different approach will be used hereafter. It was initiated by the third--named author~\cite{L01,L09} in order to provide an algorithm sufficiently efficient in the present context. The underlying method only allows one to find automatic tilings with an injective coding; however, this proves adequate for the Paper--Folding Number Wall. The overlap between coded tiles is used to ensure that an injective coding is possible. The algorithm is based on Lemmata \ref{lem:consistentOverlap} and  \ref{lem:completeTetrad} (with $d=2$) and is described in Figure \ref{fig:algo2}.

\begin{figure}[h!]
\framebox{
\begin{algorithm}[H]
 \KwData{$k$, $l$, $r$ and $S_{m,n}$ for $a\leq m \leq b$ and $c\leq n\leq \delta$}
 \KwResult{A $k$--substitution $\varphi$, a prolongable tuple $\mathbf{s}$, and an $r$--consistent $l$--coding $\tau$ such that $\tau_r\left(T\right)$ agrees with $S_{m,n}$ for $T = T_{\varphi, s}$, or failure}
 \emph{Pass 1}: Determine the coding and generate the tiling $T$  \par $x=1$\par
 \While{$\frac{a+r}{l-r}< i\leq\frac{b-r}{l-r}$ and $\frac{c+r}{l-r}< j\leq\frac{\delta-r}{l-r}$}{
   \eIf{$\exists\; 0<y<x$ such that $\tau(y)(s,t)=S_{(l-r)(i-1)+s,(l-r)(j-1)+t}$ for all $1\leq s,t\leq l$}{
     Set $T(i,j)=y$\par
   }{
     Set $T(i,j)=x$,\par
     define $\tau(x)$ by $\tau(x)(s,t)=S_{(l-r)(i+1)+s,(l-r)(j+1)+t}$ for $1\leq s,t\leq l$\par
     and increase the number of tiles by one $x=x+1$\par
   }
 }
 \emph{Pass 2}: Find a substitution $\varphi$ and a prolongable tuple $\bs$ such that $T_{(\varphi,\bs)}$ covers $T$\par
 \While{$\frac{a+r}{(l-r)k}<i\leq\frac{b-r}{(l-r)k}$ and $\frac{c+r}{(l-r)k}<j\leq\frac{\delta-r}{(l-r)k}$}{
   Either define if not already defined, or verify that the following holds for all $1\leq s,t\leq k$:\par
   $\varphi(T(i,j))(s,t)=T(k(i-1)+s,k(j-1)+t)$\par
   and return failure if it doesn't hold \par
 }
 For each  $\bo \in\{0,1\}^2$ set $s_{\bo}=T(\bo)$ in the $\bo^{th}$ orthant of $\Z^2$\par
 \emph{Pass 3}: Verify that any $2$--pattern in $T$ in the region $\left\{(i,j)\sep \frac{a+r}{(l-r)}\le i\leq\frac{b-r}{(l-r)} \text{ and } \frac{c+r}{(l-r)}\le j\leq\frac{\delta-r}{(l-r)}\right\}$ already appears in the region $\left\{(i,j)\sep \frac{a+r}{(l-r)k}<i\leq\frac{b-r}{(l-r)k} \text{ and } \frac{c+r}{(l-r)k}<j\leq\frac{\delta-r}{(l-r)k}\right\}$ \par
Return failure if this doesn't hold
\vspace{3mm}
\end{algorithm}
}
\caption{How to tile a portion of a two dimensional table.}
\label{fig:algo2}
\end{figure}

\subsection{The Paper--Folding Number Wall}

The algorithms of the previous sections enable one to detect a structure in the Number Wall of the Paper--Folding sequence:

\begin{thm}\label{thm:main}
The finite portion of the Paper--Folding Number Wall over $\bbf_3$
\begin{equation}\label{eq:finitePortion}
\left(S_{m,n}\right)_{-55\leq m \leq 2400 ,-5220\leq n\leq 5220}
\end{equation}
agrees with a tiling $T'=\tau'(T)$. This tiling is the image under a coding $\tau'$ of a substitution tiling $T=T_{\left(\varphi,(1,2,3,4)\right)}$. Here,
\begin{enumerate}
\item $\varphi$ is a $2$--substitution over $\Sigma=\left\{1,\ldots,2353\right\}$ for which tiles $1,2,3$ and $4$ are respectively $(0,0),(0,1),(1,0)$ and $(1,1)$--prolongable;
\item $\tau'$ is an $8$--coding defined with the help of a 13--coding $\tau$ by the relation
\begin{equation}\label{deftau'}
\tau'(s)(\bn)=\tau(s)(\bn+\mathbf{3})
\end{equation}
for every $s\in\Sigma$ and every $\bn\in\{1,\ldots, 8\}^2$;
\item $\tau$ is a $13$--coding defined over $\Sigma$ and taking values in $\bbf_3$. This coding is $5$--consistent for the tiling $T$ restricted to the region
\begin{equation}\label{regionR}
\mathcal{R}=[-3, 149]\times [-325, 325].
\end{equation}
\end{enumerate}
\end{thm}

\begin{rem}
The additional translation included in the definition of the coding $\tau'$ in~\eqref{deftau'} is an adjustment of a technical nature 
(see the Appendix for details). It enables one to match the zeroth row of the tiling with the finite portion of  the Paper--Folding sequence used to build part of its Number Wall. This will allow later an easy comparison between the substitution that generates the zeroth row and the substitution generating the Paper--Folding sequence outlined in Example~\ref{es:paperFolding}.
\end{rem}

Theorem~\ref{thm:main} is obtained by application of the algorithm described in Figure~\ref{fig:algo2} with the pa\-ra\-me\-ters
\begin{equation*}\label{param}
k=2,\quad l=13, \quad r=5,\quad a=-55,\quad b=2400,\quad c=-5220 \quad \textrm{and}\quad \delta=5220.
\end{equation*}

The region $\mathcal{R}$ defined in~\eqref{regionR} is obtained from Pass 3 of this algorithm: it corresponds to the range of indices $(m,n)\in\Z^2$ such that $$-3.125 \; =\; \frac{-55+5}{(13-5)\cdot 2}\; <\; m\; \le\; \frac{2400-5}{(13-5)\cdot 2}=149.6875$$ and $$-325.9375 \; =\; \frac{-5220+5}{(13-5)\cdot 2}\; <\; n\; \le\; \frac{5220-5}{(13-5)\cdot 2}\; =\; 325.9375.$$ Note that Pass 3 then guarantees that the tiling $T$ satisfies the assumptions of Lemma~\ref{lem:completeTetrad} upon setting
\begin{equation}\label{defmM}
\mathbf{m}=(-4, -326) \quad \textrm{and}\quad \mathbf{M}=(149, 325).
\end{equation}
The resulting conclusion is stated as a proposition:
\begin{prop}\label{remregion}
With the notation of Theorem~\ref{thm:main}, every 2--pattern in $T$ already appears in $T|_{(\bmf, \mathbf{M}]}$.
\end{prop}

\par
Proving that the tiling $T'=\tau'\left(T\right)$ in Theorem~\ref{thm:main} is  the Paper--Folding Number Wall and determining its deficiency require a fine analysis of the pattern of zeros appearing in the portion $\tau'(T|_{\mathcal{R}})$. This, in turn, relies on some properties of the  coding $\tau$ and of the 2--substitution $\varphi$. These properties are stated in the next two propositions and depend on two specific subsets of tiles, namely
\begin{equation}\label{setSS'}
S=\left\{1, 2, 6, 7, 12, 13, 20, 29\right\}\quad \textrm{and}\quad S'=S\cup\left\{5\right\}.
\end{equation}

\begin{prop}[Properties of the 13--coding $\tau$]\label{proptau}
The 13--coding $\tau$ introduced in Theorem~\ref{thm:main} satisfies the following property: if, for a given tile $s\in\Sigma$, the $13\times 13$ square $\tau(s)$ contains a 4--pattern comprising only zero entries, then $s\in S'$.

More precisely, under such an assumption,
\begin{itemize}
\item either $s=5$ and $\tau(s)$ is identically zero;
\item or else $s\in S$, in which case the top 9 rows of $\tau(s)$ are identically zero while all entries in the tenth row equal 1.
\end{itemize}
Furthermore, for every $s\in\Sigma$, the 13--pattern $\tau(s)$  is contained in the rectangular pattern~\eqref{eq:finitePortion}.
\end{prop}

\par
The last claim in Proposition~\ref{proptau}  follows from the construction of the coding $\tau$ in Pass 1 of the algorithm described in Figure~\ref{fig:algo2}. The rest of the proposition is established by (computer) inspection of the coding $\tau:\Sigma\to\Sigma^{\left\{1,\ldots,13\right\}^2}$, which is available at~\cite{L17} in the file \texttt{dragon\_codes\_B.dat}

\begin{prop}[Properties of the 2--substitution $\varphi$]\label{propvarphi}
The 2--substitution $\varphi$ introduced in Theorem~\ref{thm:main} satisfies the following properties:
\begin{enumerate}
\item the image of tile 5 is the $2\times 2$ square all of whose entries are 5;
\item for any tile $s\in S$,
\begin{equation}\label{s1s2}
\varphi(s) = \begin{array}{cc} 5 & 5 \\ s_1 & s_2 \end{array} \quad \textrm{with} \quad s_1, s_2\in S;
\end{equation}
\item if the image $\varphi(s)$ of a tile $s\in\Sigma$ contains a tile in $S'$, then $s\in S'$.
\end{enumerate}
\end{prop}

\par
Proposition~\ref{propvarphi} is established by (computer) inspection of the substitution $\varphi:\Sigma\to\Sigma^{\left\{1,2\right\}^2}$, which is explicit and available at~\cite{L17} in the file \texttt{dragon\_ tetrads\_B.dat}. This file actually contains the 6721 2--patterns that can be found in the tiling $T$ restricted to the region $\mathcal{R}$. The first 2353 correspond to the successive images of tiles $1, \dots, 2353\in\Sigma$ under the 2--substitution $\varphi$. Also, the values of $s_1$ and $s_2$ appearing in~\eqref{s1s2} are explicitly recorded in Figure~\ref{[dragcode]} below as they will be needed later.

\par
Together with the initial conditions
\begin{equation}\label{initcond}
T(0,0)=1,\quad T(0,1)=2, \quad T(1,0)=3 \quad \textrm{and} \quad T(1,1)=4,
\end{equation}
the 2--substitution $\varphi$ can be used to generate any finite portion of the tiling $T=T_{\left(\varphi,(1,2,3,4)\right)}$. An example of such a portion is represented in Figure~\ref{[statseg]}. Note the presence of the initial pattern
\[
\begin{array}{cc}
1 & 2 \\
3 & 4
\end{array}
\]
therein (cf.~the second and third rows).

\begin{figure}[htb!] 
\centering
{\tt \scriptsize
\begin{boxedverbatim}
  5   5   5   5   5   5   5   5   5   5   5   5   5   5   5   5   5   5   5   5   5   5
 12   6  29  20   7  13  29  20  12   6   1   2   7  13   1   2  12   6   1   2   7  13
 41  96  80  65  52  40  30  21  14   8   3   4   9  15  22  31  41   8   3   4   9  40
134 114  97  81  66  53  42  32  23  16  10  11  17  24  33  43  54  67  10  11 115 135
156 136 116  98  82  68  55  44  34  25  18  19  26  35  45  56  69  83  99 117 137 157
179 158 138 118 100  84  70  57  46  36  27  28  37  47  58  71  85 101 119 139 159 180
204 181 160 140 120 102  86  72  59  48  38  39  49  60  73  87 103 121 141 161 182 205
230 206 183 162 142 122 104  88  74  61  50  51  62  75  89 105 123 143 163 184 207 231
256 232 208 185 164 144 124 106  90  76  63  64  77  91 107 125 145 165 186 209 233 257
 70 258 234 210 105 166 146 126 108  92  78  79  93 109 127 147 167 187 211 235 259 283
 86  72 260 236 212 188 168 148 128 110  94  95 111 129 149 169 189 213 237 261 284 308
104  88 285 262 238 214 190 170 150 130 112 113 131 151 171 191 215 239 263 286 309 337
\end{boxedverbatim}
}
\caption{Portion of the tiling $T$ as described in Theorem~\eqref{thm:main}.}
\label{[statseg]}
\end{figure}

\section{Verifying the Properties of the Generated Number Wall}\label{sec:verify}

Proving that the tiling described in the statement of Theorem~\ref{thm:main} is the $\bbz^2$--tiling given by the Paper--Folding Number Wall will be done by verifying that it satisfies the Frame Constraints. From Corollary~\ref{frameconstraints}, this will imply that it is the Number Wall of the sequence that sits in its zeroth row. The claim will then follow from verifying that this sequence is the Paper--Folding sequence, thereby establishing Theorem~\ref{thmderive}. Let then $T$, $T'$, $\tau$, $\tau'$ and $\varphi$ be as in Theorem \ref{thm:main}.

\begin{thm}\label{thm:consistency}
The coding $\tau$ is 5--consistent for the tiling $T$.
\end{thm}
\begin{proof}
The assumptions of Lemma~\eqref{lem:completeTetrad} have been shown to hold in order to state Proposition~\ref{remregion}. Corollary~\ref{latest} then implies that the consistency condition needs only be verified in $T|_{(\bmf, \mathbf{M}]}$, where $\bmf$ and $\mathbf{M}$ are defined in~\eqref{defmM}. This verification has been done as part of Point 3 in Theorem~\ref{thm:main}.
\end{proof}

Since $\tau$ is a 13--coding which is 5 consistent, equation~\eqref{deftau'} implies that, with the notation of Section~\ref{sub:overlaps}, $$\tau'=\tau_{(5, \mathbf{3})}.$$ It then follows from Lemma~\ref{eq:shiftedCoding} that the images of the tiling $T$ under $\tau'$ and $\tau_{5}$ are the same up to a shift of 3, both vertically and horizontally; that is, for any $\bn\in\Z^2$,
\begin{equation}\label{reltau'tau5}
\tau'\left(T \right)(\bn) = \tau_5\left(T\right)(\bn+\mathbf{3}).
\end{equation}

This equation enables one to interpret the squares and circles in Figure~\ref{[wallseg]} and to illustrate how the substitution, coding and overlap look like. The top circle marks the origin. The top--left $13\times 13$ square is the image under $\tau$ of tile $17$, with a circle in its center.  The image of $17$ under $\varphi$ is the $2$--pattern given by the matrix
\[
\begin{array}{cc}
35 & 45 \\
47 & 58
\end{array}
\]
(this information is contained in Figure~\ref{[statseg]}). The images under $\tau$ of these four tiles are represented by four squares with a circle around their respective centers. By definition, it then follows that $\tau_5$ maps tiles 17, 35, 45, 47 and 58 to the top--left $8\times 8$ squares sitting in their respective images under $\tau$. Finally, equation~\eqref{reltau'tau5} shows that the tiling $T'=\tau'(T)$ is obtained by translating all the $8\times 8$ squares thus obtained by the vector $\mathbf{3}=(3,3)$.

\begin{thm}\label{thm:boundedDeficiency}
The zero entries with non--negative row indices in the tiling $T'=\tau'(T)$ appear in the form of squares with side lengths at most 3 and with horizontal and vertical edges.
\end{thm}

With the conventions and definitions adopted in Section~\ref{defpropnbwall}, Theorem~\ref{thm:boundedDeficiency} is thus saying that a window in the part of the tiling $T'$ with non--negative row indices has deficiency at most 4.

The proof of Theorem~\ref{thm:boundedDeficiency} requires first the following lemma, which gives some properties of the tiling $T$:

\begin{lem}\label{propT}
The tiling $T$ is such that:
\begin{enumerate}
\item all entries in rows with index at most -1 are identically equal to 5;
\item the zeroth row contains only tiles from $S$;
\item the rows with positive indices contain no entry in $S'$.
\end{enumerate}
Here, the sets of tiles $S$ and $S'$ are those defined in~\eqref{setSS'}
\end{lem}

\begin{proof}
Consider first the orthant comprising the point $(0,0)$ with the associated initial condition $T(0,0)=1$. Let $s$ be a tile in the zeroth line of this orthant. The construction rule for the tiling $T$ described in~\eqref{eq:subsTiling2} implies that the bottom two entries in $\varphi(s)$ determine two additional entries in the zeroth line of this orthant. Furthermore, all entries in the zeroth line are obtained this way. Also, all entries which are not in the zeroth row are determined, either from the top two entries in the image under $\varphi$ of a tile lying in the zeroth row, or else from any entry in the image of any other tile. Since tile 1 belongs to the set $S$, Points 1 \& 2 in Proposition~\ref{propvarphi} show, on the one hand that  the zeroth row in the orthant under consideration contains only tiles from $S$ and on the other that all rows with index at most -1 are identically equal to 5.

A similar reasoning applies to the orthant comprising the point $(0,1)$, which is mapped under $T$ to tile 2, which is an element in $S$. This establishes the first two claims in the lemma.

Consider now the orthant comprising the point $(1,1)$ with the associated initial condition $T(1,1)=4$. All elements in this orthant are obtained as entries in the image under $\varphi$ of a tile $s$ lying in this orthant. As tile 4 is not in $S'$, an easy induction based on Point 3 in Proposition~\ref{propvarphi} implies that no entry in this orthant is in $S'$.

A similar reasoning applies to the orthant comprising the point $(1,0)$, which is mapped under $T$ to tile 3, which is not an element in $S'$. This establishes the last claim in the lemma.
\end{proof}

\begin{proof}[Proof of Theorem~\ref{thm:boundedDeficiency}] The first step is to show that the part of the tiling $T'$ with non--negative row indices cannot contain any $4\times 4$ zero window.

\begin{lem}\label{lempatternzero4neg}
Let $P$ be a 4--pattern contained in $T'$ which is identically zero. Then $P$ is contained in the portion of the tiling $T'$ with negative row indices.
\end{lem}

\begin{proof}
It is an immediate consequence of equation~\eqref{reltau'tau5} that $P$ also sits as a 4--pattern in the tiling $\tau_5(T)$. Lemma~\ref{lem:consistentOverlap} applied with the parameters $l=13$, $r=5$ and $r'=4$ then implies that $P$ is contained in  the image $\tau_5(P')$ of a 1--pattern $P'$ sitting in $T$. Let then $(m,n)\in\Z^2$ be such that $P'(\bone)=T(m, n)$. From Definition~\ref{def:consistentPattern}, it holds that $\tau_5(P') = \tau(P'(\bone)) = \tau(T(m, n))$.

Since the zero 4--pattern $P$ is contained in $\tau(T(m, n))$, Proposition~\ref{proptau} implies that $T(m,n)\in S'$. Thus, from Lemma~\ref{propT}, it holds that $m\le 0$ and moreover  that $T(m,n)=5$ when $m<0$ and $T(m,n)\in S$ when $m=0$.

Note then these two observations:  on the one hand, $\tau'\left(T(m,n)\right)$ is the $8\times 8$ subsquare in the $13\times 13$ square $\tau(T(m, n))$ corresponding to the row and column indices between 4 and 11 (see equation~\eqref{deftau'}). On the other hand, it follows from  the way the tiling $\tau'(T)$ is defined in~\eqref{eq:codedTiling} that the $8\times 8$ square $\tau'\left(T(m,n)\right)$ sits between rows $-7$ and 0 if $m=0$ and has maximal row index $-8$ if $m\le -1$.

Since when $m=0$, the 4--pattern $P$ cannot overlap with rows 10 to 13 in $\tau(T(0,n))$ (cf.~Proposition~\ref{proptau}), the above two observations enable one to conclude that the pattern $P$ is contained in the region of the tiling $\tau'(T)$ with row index at most -2.
\end{proof}

In order to complete the proof of Theorem~\ref{thm:boundedDeficiency}, consider a 4--pattern $P$ in the region of the tiling $T'$ with non--negative row indices. From Lemma~\ref{lempatternzero4neg}, $P$ cannot be identically zero. Furthermore, Lemma~\ref{lem:consistentOverlap} implies here also that $P$ is contained in $\tau(T(m,n))$ for some $m, n\in\Z$. From the last claim in Proposition~\ref{proptau}, the pattern $\tau(T(m,n))$ appears in the portion of the Number Wall~\eqref{eq:finitePortion}, where all zero entries occur within squares with horizontal and vertical edges (this is Theorem~\ref{themzeroentries}).

Since there cannot be a zero 4--pattern in the region $m\ge 0$ of the tiling $T'$, the zeros in the pattern $P$ take the shape of a rectangle $\mathcal{D}$ obtained as the intersection between a square $\mathcal{C}$ with side length at most 3 and the $4\times 4$ square determined by $P$.

Assume that $\mathcal{C}$ is not entirely contained in the pattern $P$ as there is otherwise nothing more to prove. The rectangle $\mathcal{D}$  then admits a point, say $A$, lying in the interior of the $4\times 4$ square determined by $P$. Consider another 4--pattern $P'$ containing $\mathcal{D}$, one of which corners coincides with the point $A$. The same argument as previously shows, on the one hand that $P'$ is contained in the tiling $T'$ and, on the other, that all zero entries in $P'$ appear in the form of squares with side length at most $3$. This is enough to conclude that the rectangle $\mathcal{D}$ can be extended to a square with side length at most $3$ lying in the tiling $T'$.

This concludes the proof of Theorem~\ref{thm:boundedDeficiency}.
\end{proof}

\begin{thm}\label{thm:frameConstraints}
The tiling $T'$ satisfies the Frame Constraints.
\end{thm}

\begin{proof}
By Theorem \ref{thm:boundedDeficiency} the maximal side length of a window lying the in part of $T'$ with non--negative row indices is 3. Therefore, by Corollary~\ref{frameconstraints}, the Frame Constraints are determined by patterns of size at most 7 (see also Figure~\ref{[origwind]}). Applying Lemma \ref{lem:consistentOverlap} again with $l=13$, $r=5$ and $r'=7$, every $7$--pattern in $\tau_5\left(T\right)$ (and therefore in $T'=\tau'\left(T\right)$ from~\eqref{reltau'tau5}) is contained in $\tau_5(P)$ for some $2$--pattern $P$ in $T$. From Proposition~\ref{remregion},  this pattern $P$ already appears in $T|_{(\bmf, \mathbf{M}]}$.

Also, from Theorem~\ref{thm:main}, the image of $T|_{(\bmf, \mathbf{M}]}$ under $\tau_5$ is, up to a translation by the vector $\mathbf{3}=(3,3)$, the restriction of the Number Wall~\eqref{eq:finitePortion} to the region $(8\bmf, 8\mathbf{M}]$. Note that from the values of $\bmf$ and $\mathbf{M}$ in~\eqref{defmM}, this restriction followed by a translation by $\mathbf{3}$ is clearly contained in~\eqref{eq:finitePortion}. Since the Frame Constraints are satisfied in the Number Wall~\eqref{eq:finitePortion}, this completes the proof of the theorem.
\end{proof}

Theorems~\ref{thm:boundedDeficiency} and~\ref{thm:frameConstraints} can be rephrased as follows~:

\begin{coro}
The tiling $T'$ is a Number Wall with deficiency 4.
\end{coro}

From the last claim in Corollary~\ref{frameconstraints}, $T'$ is the Number Wall of the sequence sitting in its zeroth row. This sequence is now determined:

\begin{thm}\label{thm:verification}
The tiling $T'$ has the Paper--Folding sequence in its zeroth row.
\end{thm}
\begin{proof}
From Point 3 in Lemma~\ref{propT},  only the eight tiles in $S$ appear in the zeroth row of $T'$. Figure \ref{[dragcode]} tabulates these tiles, their images under the substitution $\varphi$ (in other words, the values of $s_1$ and $s_2$ in~\eqref{s1s2}) and also their images under the codings $\tau$ and $\tau'$ restricted to the zeroth line (which is the $11^{th}$ row of $\tau$ and thus the $8^{th}$ row of $\tau'$ --- this follows immediately from the definition of $\tau'$ in~\eqref{deftau'} and from the fact that, given an integer $n$, $\tau'(T(0,n))$ sits between rows $-7$ and 0 in the tiling $T'$).
For simplicity, the original eight tiles are mapped bijectively to $\left\{0,\ldots,7\right\}$ and the above restrictions of $\varphi$, $\tau$ and $\tau'$ are still denoted in the same way.

\begin{figure}[htb!]
\centering
\begin{tabular*}{0.9\textwidth}{@{\extracolsep{\fill}}|cccccc|}
\hline
  & Tile & Coded tile & Substitution $\varphi$ & Coding $\tau, \tau'$   &\bigstrut \\
\hline
  & 2  & 0 & 0\ 2 & 110\boxed{00100110}00 & \bigstrut[t] \\
  & 13 & 1 & 0\ 3 & 110\boxed{00100111}00 & \\
  & 7  & 2 & 1\ 6 & 110\boxed{00110110}00 & \\
  & 12 & 3 & 1\ 7 & 110\boxed{00110111}00 & \\
  & 20 & 4 & 4\ 2 & 111\boxed{00100110}00 & \\
  & 6  & 5 & 4\ 3 & 111\boxed{00100111}00 & \\
  & 1  & 6 & 5\ 6 & 111\boxed{00110110}00 & \\
  & 29 & 7 & 5\ 7 & 111\boxed{00110111}00 & \bigstrut[b] \\
\hline
\end{tabular*}
\caption{The generated tiling restricted to the zeroth row: substitutions and codings (the original eight tiles are mapped bijectively to $\left\{0,\ldots,7\right\}$).}
\label{[dragcode]}
\end{figure}

Using coded tiles, the zeroth row of the tiling $T'$ is thus the $\Z$--tiling $\tau'\left(T_{(\varphi, (6,0)})\right)$ (this follows from  the first two initial conditions in~\eqref{initcond}, which are expressed in the language of non--coded tiles).

Looking at the boxed segments  in Figure \ref{[dragcode]}  that represent the coding $\tau'$ restricted to the zeroth line reveals that mapping the (coded) tiles $4,5,6,7$ to $0,1,2,3$ respectively in the substitution $\varphi$  is consistent with their definition on $0,1,2,3$ (for instance, $\varphi(4) = 4\; 2$ and mapping 4 to 0 transforms the tiling $4\; 2$ to $0\; 2$, which is indeed $\varphi(0)$). Furthermore, this mapping leaves unchanged the images of the tiles under $\tau'$ (for instance, $\tau'(4)=\tau'(0)$). This implies in particular that the zeroth row of $T'$ is also the $\Z$--tiling $\tau'\left(T_{(\varphi, (2,0)})\right)$.

A further comparison with $\psi$ and $\rho$ introduced in Example \ref{es:paperFolding} shows that, in fact,
\[
\tau'(\varphi(s)) = \rho\left(\psi^4(s)\right)
\]
for each $s\in\left\{0,1,2,3\right\}$, and therefore for all $s\in\left\{0,\dots,7\right\}$ upon identifying tiles as above when needed.

It is elementary to verify that $$T_{\left(\psi^4,(2,0)\right)} = T_{\left(\psi,(2,0)\right)}.$$ Since the Paper--Folding sequence is $\rho\left(T_{(\psi,(2,0))}\right)$ (cf. Example \ref{es:paperFolding}) ,  this concludes the proof that the zeroth row of the generated Number Wall is the Paper--Folding sequence.
\end{proof}

\section{The $t$--adic Littlewood Conjecture in other Characteristics}

\subsection{Sequences with Small Deficiency over Finite Fields}

The aim of this section is to discuss to what extent the value of the deficiency appearing in Theorem~\ref{thmderive} (viz.~4) can be improved and/or generalized to other characteristics.

It is easily seen that one cannot have a Number Wall over $\F_3$ with no zero entry\footnote{in the rows with non--negative indices. This assumption will always be implicit in what follows.}. This amounts to saying that any infinite Hankel matrix over $\F_3$ admits a \emph{singular connected minor} (that is, a singular square submatrix whose row and column indices are consecutive). This prompts the following more general open problem, which the authors have not been able to solve:

\begin{qst}\label{qst1}
Let $\F_q$ be a finite field with $q\ge 2$ elements. Does there exist an integer $n(q)\ge 1$ such that any square matrix with dimensions $n(q)\times n(q)$ and with entries in $\F_q$ admits a  singular connected minor (as defined above)?
\end{qst}

In other words, Question~\ref{qst1} amounts to asking  how big a hyperinvertible matrix can be over a finite field, where hyperinvertibility of a matrix means that all \emph{connected} square submatrices are in\-ver\-ti\-ble. Note that when the field is infinite, the well--known class of Cauchy matrices provides examples of arbitrarily large hyperinvertible matrices. Also, Question~\ref{qst1} is well--understood in the case that one does not restrict oneself to the case of \emph{connected} minors: it is indeed proved in~\cite{annalee} that there is no $q\times (q-1)$ matrix over $\F_q$ whose (non--necessarily connected) minors of order $q-1$ and $q-2$ are all different from zero.

\paragraph{} On another front, Theorem~\ref{thmderive} raises the question so as to whether there exists a sequence with deficiency smaller than 4 over $\F_3$ and, possibly, with optimal value 2. This is indeed the case, and the discovery of a sequence generating a Number Wall with only isolated zeros, the \emph{Pagoda sequence} $\left(\pi_n\right)_{n\in\Z}$, was made in~\cite{L97} (see also~\cite{L09} for a fuller account; the origin of the name of the sequence is also explained in the latter reference). It is defined from the Paper--Folding sequence as follows: for any $n\in\Z$, $$\pi_n\,=\, f_{n+1}-f_{n-1}.$$ Tiling the Number Wall obtained from this sequence confirms the above claim:

\begin{thm}~\label{thmPagoda}
The Number Wall of the Pagoda sequence over $\F_3$ has only isolated zero entries.
\end{thm}

The proof of Theorem \ref{thmPagoda} proceeds along the same lines as the proof of Theorem \ref{thm:main} and will not be detailed here. 
These details can be made explicit from the codes available at~\cite{L17, N18} which deal, not only with the case of the Paper--Folding sequence, but also with that of the Pagoda sequence.
The main missing ingredient is the verification of the analogue of Theorem \ref{thm:verification} for the Pagoda sequence. This requires a comparison between two given automatic sequences, which in the case of the Paper--Folding sequence could be done explicitly due to the very simple substitution and coding that generate it. 
In view of Theorem~\ref{themdeficiency}, Theorem~\ref{thmPagoda} implies the existence of a formal Laurent power series $\Xi$ such that 
$$\inf_{N\neq 0, k\ge 0} \left|N\right|\cdot \left|\langle N t^k \Xi\rangle\right|\, = \, 2^{-2}.$$ 
This equality corresponds 
to the ``worst'' possible case when $t$--LC fails over $\F_3$.

\paragraph{} As a matter of fact, a generalisation to other characteristics is suggested by computer evidence:

\begin{conj}\label{conjfinal}
The Paper--Folding and Pagoda sequences seen as sequences over a finite field $\F_p$ have bounded deficiency 4 and 2 respectively for all prime $p\equiv 3 \pmod 4$, and unbounded deficiency for all other primes.
\end{conj}

Conjecture~\ref{conjfinal} has been checked extensively by computer inspection of the Number Walls of the sequences under consideration. For instance, in the case of the Paper--Folding sequence, it has been verified in finite $3000\times 3000$ segments of its Number Walls over $\F_p$ for all $p\le 101$. The difficulty in validating this conjecture for a \emph{given} characteristic is that the number of tiles needed to put in place the strategy exposed in Section~\ref{strategy} increases very quickly with the characteristic (for instance, preliminary attempts to construct a Pagoda Wall over $\F_7$ indicate the putative existence of a tiling made of 1.4 million tiles).

If indeed true, Conjecture~\ref{conjfinal} would imply that the $t$--adic Littlewood Conjecture fails (at least) over any field with characteristic a prime congruent to 3 modulo 4.

\subsection{On the Laurent Series of the Paper--Folding Sequence in a Field with Characteristic 2}\label{sec:charTwo}

The Hankel matrices of the Paper--Folding and Pagoda sequences are much better understood over $\mathbb{K}=\bbf_2$. This follows from the fact that the Laurent series they define  in $\bbf_2\left(\left(t^{-1}\right)\right)$ are both quadratic. Indeed, note that
$f_{2n+1} = n$ modulo 2 and $f_{2n}=f_n$ for every $n\in\bbn$. Therefore,
\begin{align*}
\Phi(t) & = \sum_{n > 0} f_n t^{-n} \\
& = \sum_{n > 0} f_{2n} t^{-2n} + t^{-1} \sum_{n \geq 0} f_{2n+1} t^{-2n} \\
& = \sum_{n > 0} f_n t^{-2n} + t^{-1} \sum_{n \geq 0} n t^{-2n} \\
& = \Phi\left(t^2\right) - t^{-3} \sum_{n\ge 1} t^{-4n}\\
& = \Phi\left(t^2\right) - \frac{t^{-3}}{1 - t^{-4}}\cdotp
\end{align*}
This implies that the Laurent series $\Phi(t)$ satisfies the following quadratic equation over $\bbf_2\left(t\right)$:
\[
\Phi\left(t\right)^2 + \Phi(t) + \frac{t}{1 + t^4} = 0.
\]

Similarly for $\Pi(t) = \sum_{n > 0} \pi_n t^{-n} $,  one obtains from the relations  $f_0=f_1=0$ that
\[
\Pi\left(t\right) = \sum_{n > 0} \left(f_{n+1} - f_{n-1}\right) t^{-n} = \left(t-t^{-1}\right)\Phi(t).
\]
Thus, $\Pi(t)$  satisfies the quadratic equation
\[
\Pi(t)^2+\frac{1+t^2}{t}\Pi(t)+\frac{1}{t}=0.
\]
In the same paper as where the $p$--adic Littlewood Conjecture was first stated \cite{MT04}, De Mathan and Teuli\'e established that $t$--LC holds for quadratic irrational power series. This is therefore in particular the case for the series $\Phi$ and $\Pi$ above. Of course, this claim extends naturally to any field with characteristic 2.

It should also be noted that much more is currently known about the occurrence of windows in the Number Walls of quadratic series --- see for further details the paper by Kemarsky, Paulin and Shapira~\cite{KPS17}, which uses dynamics on Bruhat--Tits tree. The connections to the present work are explained in~\cite[p.5]{KPS17}.

\newpage

\section*{Appendix~: Implementation and Code}\label{sec:implementation}
\addcontentsline{toc}{section}{Appendix~: Implementation and Code}

It is worthwhile to mention several matters related to the implementation of the algorithms des\-cri\-bed in Figures \ref{fig:algo1} and \ref{fig:algo2}. The complete Magma program is contained in the file \texttt{dragon\_wall\_B.mag} available at~\cite{L17}. The results have been confirmed by an independent implementation in Sage,
for which the program is available at \cite{N18} (\footnote{It should be noted that both in~\cite{L17} and in~\cite{N18}, the \emph{Paper--Folding sequence} is referred to as the \emph{Dragon sequence}.}).

\subsection*{The Wall Builder}\label{sec:tileBuilder}
The wall builder described in Figure~\ref{fig:algo1} is implemented as a Magma program
{\tt procedure NumberWall ($\sim$seq, mlo, mhi, nlo, nhi, $\sim$wal)},
where {\tt seq} and {\tt wal} hold respectively the sequence $S_{0, n}$ and the wall entries $S_{m,n}$,
and {\tt mlo} and {\tt mhi} (resp.~{\tt nlo} and {\tt nhi}) the ranges of $m$ (resp.~of $n$) in the output segment.
Initial rows $m \le 0$ are determined via Definition \ref{defnumberwall}. The
variable number ${\tt mlo}$ of rows of the top infinite zero part is asserted to be less than or equal to $-2$,
and serves as a sentinel window that allows conveniently for subsequent tiling implementation requirements.
\par

The natural boundary of a wall segment with, say,
$0 \le m \le m^{\text{hi}}$ and $0 \le n \le n^{\text{hi}}$ is trapezoidal,
descending from a sequence segment of length $2m^{\text{hi}} + n^{\text{hi}}$
at $m = 0$ to length $n^{\text{hi}}$ at $m = m^{\text{hi}}$ (see also~\S\ref{buildingfinite} for details).
To avoid complicated program logic where square window frames cross
the boundary, two further effectively infinite sentinel windows are attached
along left and right segment edges.
This permits filling a rectangular boundary with lengths roughly
$m^{\text{hi}}\times\left(2m^{\text{hi}} + n^{\text{hi}}\right)$,
partially containing spurious entries which can ultimately be ignored.
Finally, after pruning to lengths $m^{\text{hi}}\times n^{\text{hi}}$, the rectangle contains
only valid entries.

\subsection*{Counting Deficiencies in a Wall}\label{sec:countingDeficiencies}
\sloppy A separate Magma program
{\tt procedure WallDeficiencies ($\sim$wal, mlo, mhi, nlo, nhi, $\sim$mulset)}
computes as a multiset {\tt mulset} the number of windows of each deficiency
$\delta=g+1$ (where $g$ is the side length of a window). Broken windows (where the pane crosses the boundary of the segment)
are represented temporarily by $\delta = -1$.
Not all cases with $\delta \in \{-1, 1\}$ are recorded; the purpose is to accurately detect
all visible windows occurring properly within the region. The algorithm utilises
a subset of the method discussed in Section~\ref{sec:tileBuilder}.

\subsection*{Finding Patterns in Sequences and Tilings}\label{sec:PlanarTiling}
The tile builder described in Figure~\ref{fig:algo2} is implemented as
{\tt procedure SquareTiling ($\sim$tab, mlo, mhi, nlo, nhi, tel, cid,
$\sim$codes, $\sim$tetrads, $\sim$stab)} \footnote{In the codes~\cite{L17, N18}, a \emph{tile} is referred to as a \emph{state} and a \emph{substitution} as an \emph{inflation}.}. Here, with the notation of Theorem~\ref{thm:main}, 
\begin{itemize}
\item {\tt tab} holds the wall $S_{m,n}$;
\item {\tt codes} returns the images $\tau(s)$ of the tiles $s\in\Sigma$ under the 13--coding $\tau$;
\item {\tt stab} returns the tiling $T$ restricted to the region $\mathcal{R}$;
\item {\tt tetrads} returns the images under the 2--substitution $\varphi$ of the 2353 tiles $s\in\Sigma$ followed by the rest of the 2--patterns in {\tt stab} that are not obtained as such images;
\item {\tt tel}, which stands for \emph{tile edge length}, is the value taken by $l-1$, where $l$ is the parameter used to refer to the coding $\tau$ as an $l$--coding (thus, {\tt tel}=$12$ for the coding $\tau$ in Theorem~\ref{thm:main}). In practice, the parameter {\tt tel} is restricted to even values;
\item {\tt cid} is the distance between the centers of two overlapping encoded tiles $\tau(s)$ and $\tau(s')$ in the tiling $T'$ (from Figure~\ref{[wallseg]} for instance, {\tt cid}=8 in the case under consideration). In practice, the parameter {\tt cid} is also restricted to even values.
\end{itemize}
Note that the value of the parameter $r$ when referring to the coding $\tau$ as being $r$--consistent is then $r={\tt tel}-{\tt cid}+1$ (hence, $r=5$ in Theorem~\ref{thm:main}). It is the width of overlap between two adjacent encoded tiles.

\subsection*{Adjusting some Parameters}\label{sec:parameters}

As the 2--dimensional substitution $\varphi$ is meant to generate the Paper--Folding sequence along the zeroth row of the Number Wall, and as the Paper--Folding sequence can be obtained as a one dimensional 2--substitution followed by a coding --- see Example~\ref{es:paperFolding} ---, it is natural to impose that $\varphi$ should be a $k$--substitution with $k=2$.

\paragraph{}
The translation by the vector $\mathbf{3}=(3,3)$ in the definition of the coding $\tau'$ in~\eqref{deftau'} is due to the fact that the algorithm described in Figure~\ref{fig:algo2} is actually run in~\cite{L17, N18} with a slightly different definition of the coding $\tau_r$ defined in~\eqref{eq:defnTauTag}; namely, with the coding
\begin{equation}\label{deftau_rbis}
\tilde{\tau}_r(s) = \tau(s)|_{\{u, \dots, v\}^d}.
\end{equation}
Here, $s$ is a tile in $\Sigma$ and  $$u=\left \lfloor\frac{l+1}{2}\right\rfloor -\left\lceil \frac{l-r}{2}\right\rceil +1 \qquad \mbox{ and } \qquad v=\left \lfloor\frac{l+1}{2}\right\rfloor +\left\lfloor \frac{l-r}{2}\right\rfloor$$ are integers ($\lfloor x \rfloor$ denotes the integer part of a real number $x$). Note that when $l=13$ and $r=5$ as in Theorem~\ref{thm:main}, $u=4$ and $v=11$. Definition~\eqref{deftau_rbis} ensures that the center of the square with side length $l-r$ determined by $\{u, \dots, v\}^d$ coincides with a point nearest to the center of the square $ \tau(s)$ it is contained in (which square has side length $l$). From Lemma~\ref{eq:shiftedCoding}, the resulting tiling $\tilde{\tau}_r(T)$  differs from the tiling $\tau_r(T)$  by a translation by the vector $\mathbf{j}\in\Z^d$ all of whose components are equal to $u-1$.

In the case of Theorem~\ref{thm:main}, the coding~\eqref{deftau_rbis} presents the advantage of providing a tiling $\tilde{\tau}_r(T)$ which coincides exactly with the Number Wall under consideration. The theory developed in Section~\ref{sub:overlaps} with the coding $\tau_r$ can nevertheless be seen as more robust. Indeed, on the one hand, the choice of the top--left square considered in the definition of $\tau_r$ in~\eqref{eq:defnTauTag} is more natural at least in the case of $\N^2$--tilings. On the other, the choice of a point nearest to the center of a square is not necessarily unique depending on the parity of $l$ and $r$ whereas the top--left subsquare with a predefined side length  is always unambiguously defined, given a square.

\subsection*{Canonical Order in a Two Dimensional Array}\label{sec:canonical}

One matter of a purely algorithmic nature is the ordering of the tiles in the set $\Sigma=\{1, \dots, 2353\}$. In Pass~1 of the algorithm described in Figure~\ref{fig:algo2}, not much importance is given to this ordering. However, in practice, imposing a certain canonical label to the elements of $\Sigma$ is of crucial importance in order to check the correctness of the algorithms and to interpret and exploit their outputs. Due to its theoretical irrelevance, this step has not been included in Figure~\ref{fig:algo2} (it should otherwise be considered as the last Pass of the algorithm).

A general method of assigning a canonical index to a given tile in $\Sigma$ is based on the ``matrix Manhattan metric'' and is applicable to symmetric segments in any dimension. This method is described here in dimension two in the context of Theorem~\ref{thm:main} with the notation introduced therein.

Imposing that the parameters {\tt tel} and {\tt cid} introduced in Section~\ref{sec:PlanarTiling} should be even gua\-ran\-tees that the squares $\left\{\tau(s)\right\}_{s\in\Sigma}$ which, from~Proposition~\ref{proptau}, appear in the tiling $T'=\tau'(T)$, are centered at integer points in the plane (this follows from the definition of the coding $\tau'$ in~\eqref{deftau'} and of the construction rule~\eqref{eq:codedTiling}).

Consider the portion of the tiling lying in the region defined by the conditions
\begin{equation}\label{lateregion}
m^{\text{lo}}\le m\le m^{\text{hi}} \qquad \textrm{and}\qquad n^{\text{lo}}\le n\le n^{\text{hi}}
\end{equation}
and let there be a square $\mathcal{C}$ obtained as the image under $\tau$ of a tile in $\Sigma$. Note that from the injectivity of the coding $\tau$ (this is guaranteed by the algorithm in Figure~\ref{fig:algo2} --- see also Section~\ref{sec:tilingAlgo}), the underlying tile is uniquely determined by $\mathcal{C}$. This tile is then assigned a canonical value defined as
\begin{equation}\label{eq:indexing}
\Delta = {\rm min\ }{\tt dist}(m, n) ,
\end{equation}
where
$${\tt dist}(m, n) = |m| + |n| + \frac{m}{10b} + \frac{n}{10bc}$$
with
$$b = m^{\text{hi}} - m^{\text{lo}} \qquad \textrm{and} \qquad c = n^{\text{hi}} - n^{\text{lo}}, $$
and where the minimum in \eqref{eq:indexing} is taken over all centers $(m,n)$ of squares identical to $\mathcal{C}$ lying in the region~\eqref{lateregion}.

It is elementary to establish that $\Delta$ is independent of the region~\eqref{lateregion},
provided this region should be sufficiently large along each axis.
The values assigned in \eqref{eq:indexing} allow reordering the tiles according to their ``earliest'' appearance in $T$, first by least Manhattan (or $\ell_1$) distance from the origin, then by row $m$, and finally by column $n$ (utilising floating--point arithmetic for this purpose avoids
integer overflow for large tables). This ordering then defines the index of the tile in $\Sigma=\{1, \dots, 2353\}$.
 Figure \ref{[statseg]} shows an example.
\par

\subsection*{Code Design}
Primary design targets for the present program were robustness and portability,
rather than optimum deployment of computer time or space, leading to several
similar instances of deliberately inefficient consumption of resources.
One is the preliminary rectangular wall created and then pruned by the wall
builder in Section~\ref{sec:tileBuilder}.
In similar vein, the tile builder requires room for tiles involving the initial
sequence to form correctly, which tiles are created by padding the wall with rows from the top
sentinel window: in practice, the choice of a parameter $${\tt mlo} \le -\frac{5}{2}\cdot ({\tt cid} + {\tt tel})$$ proves adequate for this purpose.
Another non--optimal step is the naive machinery selecting coordinates $(m,n)$ of centers of encoded tiles
and converting those to Magma array addresses $(i,j)$ (which start from 1).
The ensuing time penalty remains negligible,
but the improved transparency is significant for maintainability, and
for establishing the extent to which the program satisfies its specification.
\par

Similar considerations prompt limitations on the Number Wall and tiling models
employed above, mostly motivated by maintaining symmetry, but easily relaxable
at the cost of some increase in complexity. As mentioned in Section~\ref{sec:numberWall},
using Hankel determinants instead of Toeplitz determinants in the definition of a Number Wall would introduce asymmetry between $m,n$ into the
Frame Theorem (Corollary~\ref{frameconstraints}), thereby complicating its statement and application.
Also, allowing odd tile edge length {\tt tel} or odd centre separation {\tt cid} (see Section~\ref{sec:PlanarTiling})
 would require processing half--integer coordinates for the centers of encoded tiles.
Note also that the Paper--Folding sequence has been extended to a doubly--infinite sequence in Section~\ref{strategy} in view of the following observation: one--sided sequences would introduce boundaries in their Number Wall, which would require special
consideration. In contrast, since the actual data objects involved are finite
segments of walls, such boundaries are in practice automatically
incorporated at the computational stage.
\par

One case where this minimalist philosophy has been abandoned concerns the canonical ordering of tiles as described in Section \ref{sec:canonical}. Assigning these required extra program to build a `mini--wall' of
tiles, as well as an extra stage to sort the output. But the enhanced tile
builder no longer needs time--expensive searches for tile encodings when building
the substitution $\varphi$ (cf.~Pass 2 in Figure \ref{fig:algo2}),
nor does it now require simultaneous access to the
entire Number Wall. The result is doubled speed and (in concert with a rolling
row--by--row wall builder) a potentially significant reduction in space,
besides the improvement in usability which constituted its primary motivation.
\par

\paragraph{}
\renewcommand{\abstractname}{Acknowledgements}
\begin{abstract}
EN wishes to thank Dong Han Kim for discussing Remark \ref{rem:normalIndices} and related matters with him, to Nishant Chandgotia for a discussion about Section \ref{sec:tilings}, and to Maynooth University for his two visits there during the time he was working on this project. EN's research was supported by EPSRC Programme Grant: EP/J018260/1. FA's research was supported by the same grant and by EPSRC Grant  EP/T021225/1. FA would like to thank Professors David Bressoud
and Jim Propp for correspondences which helped set this project on the right track.

The authors wish to thank the referees for their careful reading of the paper and for suggestions which helped improve its quality, and for investing time in experimenting with our computer program.\\

\emph{Le premier auteur tient \`a d\'edier ce travail \`a Christian Potier en reconnaissance de ses constants encouragements et de sa passion pour les math\'ematiques.}
\end{abstract}




\bibliographystyle{amsplain}

\providecommand{\bysame}{\leavevmode\hbox to3em{\hrulefill}\thinspace}
\providecommand{\MR}{\relax\ifhmode\unskip\space\fi MR }
\providecommand{\MRhref}[2]{%
  \href{http://www.ams.org/mathscinet-getitem?mr=#1}{#2}
}
\providecommand{\href}[2]{#2}

\bibliography{bibliography}

\providecommand{\bysame}{\leavevmode\hbox to3em{\hrulefill}\thinspace}
\providecommand{\MR}{\relax\ifhmode\unskip\space\fi MR }
\providecommand{\MRhref}[2]{%
  \href{http://www.ams.org/mathscinet-getitem?mr=#1}{#2}
}
\providecommand{\href}[2]{#2}
\begin{thebibliography}{10}

\bibitem{AB07}
Boris Adamczewski and Yann Bugeaud, \emph{On the {L}ittlewood conjecture in
  fields of power series}, Probability and number theory---{K}anazawa 2005,
  Adv. Stud. Pure Math., vol.~49, Math. Soc. Japan, Tokyo, 2007, pp.~1--20.

\bibitem{APWW98}
Jean-Paul Allouche, Jacques Peyri\`ere, Zhi~Xiong Wen, and Zhi~Ying Wen,
  \emph{Hankel determinants of the {T}hue-{M}orse sequence}, Ann. Inst. Fourier
  (Grenoble) \textbf{48} (1998), no.~1, 1--27.

\bibitem{AS03}
Jean-Paul Allouche and Jeffrey~O. Shallit, \emph{Automatic sequences},
  Cambridge University Press, Cambridge, 2003, Theory, applications,
  generalizations.

\bibitem{B}
Dmitry Badziahin, \emph{Continued fractions of certain {M}ahler functions},
  Acta Arith. \textbf{188} (2019), 53--81.

\bibitem{BBEK15}
Dzmitry Badziahin, Yann Bugeaud, Manfred Einsiedler, and Dmitry Kleinbock,
  \emph{On the complexity of a putative counterexample to the {$p$}-adic
  {L}ittlewood conjecture}, Compos. Math. \textbf{151} (2015), no.~9,
  1647--1662.

\bibitem{BZ14}
Dzmitry Badziahin and Evgeniy Zorin, \emph{Thue-{M}orse constant is not badly
  approximable}, Int. Math. Res. Not. IMRN (2015), no.~19, 9618--9637.

\bibitem{B64}
Alan Baker, \emph{On an analogue of {L}ittlewood's {D}iophantine approximation
  problem}, Michigan Math. J. \textbf{11} (1964), 247--250.

\bibitem{B14}
Yann Bugeaud, \emph{Around the {L}ittlewood conjecture in {D}iophantine
  approximation}, Num\'ero consacr\'e au trimestre ``{M}\'ethodes
  arithm\'etiques et applications'', automne 2013, Publ. Math. Besan\c{c}on
  Alg\`ebre Th\'eorie Nr., vol. 2014/1, Presses Univ. Franche-Comt\'e,
  Besan\c{c}on, 2014, pp.~5--18.

\bibitem{BM08}
Yann Bugeaud and Bernard de~Mathan, \emph{On a mixed {L}ittlewood conjecture in
  fields of power series}, Diophantine analysis and related fields---{DARF}
  2007/2008, AIP Conf. Proc., vol. 976, Amer. Inst. Phys., Melville, NY, 2008,
  pp.~19--30.

\bibitem{BH14}
Yann Bugeaud and Guo-Niu Han, \emph{A combinatorial proof of the non-vanishing
  of {H}ankel determinants of the {T}hue-{M}orse sequence}, Electron. J.
  Combin. \textbf{21} (2014), no.~3, Paper 3.26, 17.

\bibitem{BHWY}
Yann Bugeaud, Guo-Niu Han, Zhi-Ying Wen, and Jia-Yan Yao, \emph{Hankel
  determinants, {P}ad\'e approximations, and irrationality exponents}, IMRN
  (2016), no.~5, 1467--1496.

\bibitem{B67}
R.~T. Bumby, \emph{On the analog of {L}ittlewood's problem in power series
  fields}, Proc. Amer. Math. Soc. \textbf{18} (1967), 1125--1127.

\bibitem{C13}
Michael Coons, \emph{On the rational approximation of the sum of the
  reciprocals of the {F}ermat numbers}, Ramanujan J. \textbf{30} (2013), no.~1,
  39--65.

\bibitem{CV12}
Michael Coons and Paul Vrbik, \emph{An irrationality measure for regular
  paperfolding numbers}, J. Integer Seq. \textbf{15} (2012), no.~1, Article
  12.1.6, 10.

\bibitem{C67}
Thomas~W. Cusick, \emph{Littlewood's {D}iophantine approximation problem for
  series}, Proc. Amer. Math. Soc. \textbf{18} (1967), 920--924.

\bibitem{C71}
\bysame, \emph{Lower bound for a {D}iophantine approximation function},
  Monatsh. Math. \textbf{75} (1971), 398--401.

\bibitem{DL63}
Harold Davenport and Donald~John Lewis, \emph{An analogue of a problem of
  {L}ittlewood}, Michigan Math. J. \textbf{10} (1963), 157--160.

\bibitem{MT04}
Bernard de~Mathan and Olivier Teuli\'e, \emph{Probl\`emes diophantiens
  simultan\'es}, Monatsh. Math. \textbf{143} (2004), no.~3, 229--245.

\bibitem{ELM17}
Manfred Einsiedler, Lindenstrauss Elon, and Mohammadi Amir, \emph{Diagonal
  actions in positive characteristic}, Duke Math. J. \textbf{169} (2020),
  no.~1, 117--175.

\bibitem{EKL}
Manfred Einsiedler, Anatole Katok, and Elon Lindenstrauss, \emph{Invariant
  measures and the set of exceptions to {L}ittlewood's conjecture}, Ann. of
  Math. (2) \textbf{164} (2006), no.~2, 513--560.

\bibitem{EK07}
Manfred Einsiedler and Dmitry Kleinbock, \emph{Measure rigidity and {$p$}-adic
  {L}ittlewood-type problems}, Compos. Math. \textbf{143} (2007), no.~3,
  689--702.

\bibitem{G59}
Felix~Ruvimovich Gantmacher, \emph{The theory of matrices. {V}ol. 1}, AMS
  Chelsea Publishing, Providence, RI, 1998, Translated from the Russian by K.
  A. Hirsch, Reprint of the 1959 translation.

\bibitem{hu}
Yining Hu and Guoniu Wei-Han, \emph{On the automaticity of the hankel
  determinants of a family of automatic sequences}, Theor.~Comput.~Sci.
  \textbf{715} (2018), 154--164.

\bibitem{KTW99}
Teturo Kamae, Jun-ichi Tamura, and Zhi-Ying Wen, \emph{Hankel determinants for
  the {F}ibonacci word and {P}ad\'e approximation}, Acta Arith. \textbf{89}
  (1999), no.~2, 123--161.

\bibitem{KPS17}
Alexander Kemarsky, Fr\'ed\'eric Paulin, and Uri Shapira, \emph{Escape of mass
  in homogeneous dynamics in positive characteristic}, J. Mod. Dyn. \textbf{11}
  (2017), 369--407.

\bibitem{K91}
Takao Komatsu, \emph{Extension of {B}aker's analogue of {L}ittlewood's
  {D}iophantine approximation problem}, Kodai Math. J. \textbf{14} (1991),
  no.~3, 335--340.

\bibitem{annalee}
Anna Lee, \emph{\"uber einige {E}xtremalaufgaben bez\"uglich endlicher
  {K}\"orper}, Acta Math. Acad. Sci. Hungar. \textbf{13} (1962), 235--243.

\bibitem{L06}
Elon Lindenstrauss, \emph{Invariant measures and arithmetic quantum unique
  ergodicity}, Ann. of Math. (2) \textbf{163} (2006), no.~1, 165--219.
  \MR{2195133}

\bibitem{MLothaire}
M.~Lothaire, \emph{Applied combinatorics on words}, Encyclopedia of Mathematics
  and its Applications, vol. 105, Cambridge University Press, Cambridge, 2005,
  A collective work by Jean Berstel, Dominique Perrin, Maxime Crochemore, Eric
  Laporte, Mehryar Mohri, Nadia Pisanti, Marie-France Sagot, Gesine Reinert,
  Sophie Schbath, Michael Waterman, Philippe Jacquet, Wojciech Szpankowski,
  Dominique Poulalhon, Gilles Schaeffer, Roman Kolpakov, Gregory Koucherov,
  Jean-Paul Allouche and Val\'erie Berth\'e, With a preface by Berstel and
  Perrin.

\bibitem{L97}
W.~Fred Lunnon, \emph{Pagodas and sackcloth: Ternary sequences of considerable
  linear complexity},
  \url{http://citeseerx.ist.psu.edu/viewdoc/summary?doi=10.1.1.39.3122}, 1997.

\bibitem{L01}
\bysame, \emph{The number-wall algorithm: an {LFSR} cookbook}, J. Integer Seq.
  \textbf{4} (2001), no.~1, Article 01.1.1, 38.

\bibitem{L09}
\bysame, \emph{The {P}agoda sequence: a ramble through linear complexity,
  number walls, {D}0{L} sequences, finite state automata, and aperiodic
  tilings}, Proceedings {I}nternational {W}orkshop on {T}he {C}omplexity of
  {S}imple {P}rograms, Electron. Proc. Theor. Comput. Sci. (EPTCS), vol.~1,
  EPTCS, [place of publication not identified], 2009, pp.~130--148.

\bibitem{L17}
\bysame, \emph{{D}ragon wall tiling program and data},
  \url{https://github.com/FredLunnon/dragon_wall/}, 2017, Github3 Repository.

\bibitem{M-FvdP81}
Michel Mend\`es~France and Alfred~Jacobus van~der Poorten, \emph{Arithmetic and
  analytic properties of paper folding sequences}, Bull. Austral. Math. Soc.
  \textbf{24} (1981), no.~1, 123--131.

\bibitem{N18}
Erez Nesharim, \emph{$t$-adic littlewood in $f_q((\frac{1}{t}))$, number wall},
  \url{https://cocalc.com/share/2317d2a7-6976-4548-a618-2869c2de6c7d/Number%20wall.sagews?viewer=share},
  2018, CoCalc Collaborative Calculation in the Cloud.

\bibitem{Q09}
Martine Queff\'elec, \emph{An introduction to {L}ittlewood's conjecture},
  Dynamical systems and {D}iophantine approximation, S\'emin. Congr., vol.~19,
  Soc. Math. France, Paris, 2009, pp.~127--150.

\bibitem{oS87}
Olivier Salon, \emph{Suites automatiques \`a multi-indices et alg\'ebricit\'e},
  C. R. Acad. Sci. Paris S\'er. I Math. \textbf{305} (1987), no.~12, 501--504.

\bibitem{jS}
Jeffrey~O. Shallit, \emph{Remarks on inferring integer sequences},
  \url{https://cs.uwaterloo.ca/~shallit/Talks/infer.ps}.

\bibitem{ST12}
Klaus Sutnera and Sam Tetruashvili, \emph{Inferring automatic sequences},
  \url{http://www.cs.cmu.edu/~sutner/papers/auto-seq.pdf}, 2012.

\bibitem{dT04}
Dinesh~S. Thakur, \emph{Function field arithmetic}, World Scientific Publishing
  Co., Inc., River Edge, NJ, 2004.

\end{thebibliography}

\vspace{8mm}

\begin{minipage}{0.4\textwidth}
\textbf{Faustin Adiceam}\\
School of Mathematics\\
The University of Manchester\\
Alan Turing Building\\
Manchester, M13 9PL\\
United Kingdom\\
\url{faustin.adiceam@manchester.ac.uk}\\
\end{minipage}
\hspace{8ex}
\begin{minipage}{0.4\textwidth}
\textbf{Erez Nesharim}\\
Department of Mathematics\\
University of York\\
Heslington\\
York YO10 5DD\\
United Kingdom\\
\url{erez.nesharim@york.ac.uk}\\
\end{minipage}

\vspace{3mm}

\begin{minipage}{0.4\textwidth}
\textbf{Fred Lunnon}\\
Department of Computer Science\\
Maynooth University\\
Co.~Kildare\\
Ireland\\
\url{fred@cs.may.ie}\\
\end{minipage}

\end{document}